\documentclass[review]{siamltex}
\usepackage{latexsym, graphicx, epsfig, amsmath, amsfonts,amssymb,subfigure}
\usepackage{multirow}
\usepackage{color}
 \usepackage[percent]{overpic}

  \newtheorem{assumption}[theorem]{Assumption}
\usepackage{algorithm}
\usepackage{algorithmicx}
\usepackage{algpseudocode}
\floatname{algorithm}{Algorithm}
\usepackage{threeparttable}
\usepackage{booktabs}

\def\ds{\displaystyle}
\def\mb{\mathbf}
\def\R{\mathbb{R}}

\title{Adaptive multi-fidelity polynomial chaos approach to Bayesian inference in inverse problems}%

\author{Liang Yan\thanks{Department of Mathematics,  Southeast University, Nanjing, 210096, China (yanliang@seu.edu.cn). This author's work is partially supported by NSF of China (No.11771081), Qing Lan project of Jiangsu Province and the Southeast University's  Zhishan Young Scholars Program. }
          \and Tao Zhou\thanks{LSEC, Institute of Computational Mathematics, Academy of Mathematics and Systems Science, Chinese Academy of Sciences, Beijing 100190, China (tzhou@lsec.cc.ac.cn). This author's work is partially supported by the NSF of China (under grant numbers 11688101, 91630203, 11571351, and 11731006), the science challenge project (No. TZ2018001), NCMIS, and the youth innovation promotion association (CAS). }}
\begin{document}
\graphicspath{{figure/}}
\maketitle

\begin{abstract}
The polynomial chaos (PC) expansion has been widely used as a surrogate model in the Bayesian inference to speed up the Markov chain Monte Carlo (MCMC) calculations. However, the use of a PC surrogate introduces the modeling error, that may severely distort the estimate of the posterior distribution.  This error can be corrected by increasing the order of the PC expansion, but the cost for building the surrogate may increase dramatically.  In this work, we seek to address this challenge by proposing an adaptive procedure to construct a multi-fidelity PC surrogate. This new strategy combines (a large number of) low-fidelity surrogate model evaluations and (a small number of) high-fidelity model evaluations, yielding an adaptive multi-fidelity approach. Here the low-fidelity surrogate is chosen as the prior-based PC surrogate, while the high-fidelity model refers to the true forward model. The key idea is to construct and refine the multi-fidelity approach over a sequence of samples adaptively determined from data so that the approximation can eventually concentrate on the posterior distribution. We illustrate the performance of the proposed strategy through two nonlinear inverse problems. It is shown that the proposed adaptive multi-fidelity approach can improve significantly the accuracy, yet without a dramatic increase in  computational complexity. The numerical results also indicate that our new algorithm can enhance the efficiency by several orders of magnitude compared to a standard MCMC approach using only the true forward model.

\end{abstract}


\pagestyle{myheadings}
\thispagestyle{plain}
\markboth{LIANG YAN AND TAO ZHOU}
{ MULTIFIDELITY GPC FOR BAYESIAN INVERSE PROBLEMS}

\section{Introduction}
Inverse problems are found in many applications, such as heat transfer, geophysics, and medical imaging \cite{Evans+Stark2002, Kaipio+Somersalo2005}. One of the central questions in inverse problems  is to estimate the unknown parameters or inputs from a set of observations. Solutions to inverse problems are subject to many potential sources of error introduced by approximate mathematical models, noisy data, and limitations in the number of observations; thus it is important to include an assessment of the uncertainties as part of the solution. The Bayesian approach can provide a systematic framework for quantifying the uncertainty in the parameter estimation for inverse problems \cite{Kaipio+Somersalo2005,Stuart2010}.  In the Bayesian framework, the prior knowledge of the unknown parameters and the forward solver are combined to yield a {\it posterior} probability distribution of the model parameters. In this way, the unknown parameters can be characterized by their posterior distributions. Since the posterior is typically not of analytical form and cannot be easily interrogated, many numerical approaches such as Markov chain Monte Carlo (MCMC) methods \cite{Brooks2011} have been \textcolor{black}{applied}.  The major challenge of MCMC is the computational burden induced by the  repeated evaluations of the forward model;  when the model is computationally expensive, the Bayesian approach can quickly become computationally prohibitive.

Computational requirements of MCMC can be reduced by solving a related computationally cheap approximated model. This idea is implemented by surrogate modeling: rather than using the full-order or high-fidelity model of forward problems, one  constructs a surrogate of the high-fidelity model, and then to sample from the posterior distribution induced by this surrogate model. Specifically,  the surrogate or low-fidelity model is often constructed offline and subsequently used on-line when running the MCMC algorithm. If the surrogate is computationally less expensive, using it in MCMC can dramatically speed up the computational procedure.  Many surrogate modeling methods have been successfully employed in this context, for example,  projection-type reduced order models \cite{Arridge2006,Frangos+Marzouk+Willcox2010,Galbally+Fidkowski+Willcox+Ghattas2010,Jin2008fast,Lieberman+Willcox+Ghattas2010},  polynomial chaos (PC) \cite{Ma+Zabaras2009,Marzouk+Najm2009,Marzouk+Najm+Rahn2007,Marzouk+Xiu2009,yan+guo2015} and Gaussian process regression \cite{kennedy2001,Rasmussen2003,stuart+teckentrup2016}, to name a few.  Very recently, the theoretical analysis shows that if  the surrogate converges at a certain rate in the prior-weighted $L_2$ norm, then the posterior distribution generated by the approximation converges to the true posterior  at least two times faster \cite{yan+guo2015,Yan+Zhang2017IP}.  However,
constructing a sufficiently accurate surrogate model over the support of the prior distribution, may not be possible in many practical problems \cite{Li+Marzouk2014SISC}. In particular, when the posterior differs significantly from the prior, constructing such a globally accurate surrogate for high-dimensional, nonlinear problems  can, in fact, be a formidable task. On the other hand, MCMC with surrogate models requires that the surrogates doe not introduce large modeling errors. Or in other words, the surrogate posterior  defined by the low-fidelity may be very different from the 'exact' posterior  defined by the high-fidelity model. To improve this, one can use two-stage MCMC simulations that leverage low-fidelity models to speed up the sampling as in\cite{Christen2005markov,Efendiev2006preconditioning}, or one can continually refine the surrogate during MCMC sampling as in \cite{Conrad2018parallel,Conrad2016JASA,Cui2014data,Li+Marzouk2014SISC}. Another possible approach to improve this is to introduce an additional term that quantifies the error between the surrogate and the forward model, see e.g., \cite{Kaipio+Somersalo2005,manzoni2016accurate}. In this approach, the modeling error is modeled as an additive noise term in the Bayesian formulation of the inverse problem, and a low-cost predictor model is constructed using Monte Carlo sampling or statistical learning.

Different from previous works, we shall develop in this work an adaptive multi-fidelity PC (AMPC) approach for Bayesian inverse problems. The motivation for this study is two folds. On one hand, it is known \cite{lu2015JCP} that if the data contain information beyond what is assumed in the prior, then the \textcolor{black}{conventional PC-based approach} can lead to large errors or require an excessive computational cost. On the other hand, if the PC surrogate is constructed over the whole prior distribution, the accuracy of the surrogate posterior can generally not be guaranteed. To solve the above problems, our strategy is as follows:
\begin{itemize}
\item To enhance the accuracy and efficiency of the PC surrogate, we shall construct a multi-fidelity PC surrogate by combining  (possibly a much large number of) low-fidelity PC surrogate evaluations and ( a small number of) true model evaluations. Similar ideas have been proposed in different applications such as in \cite{Hampton2018practical,Ng2012multi,Palar2016multi,Peherstorfer2016survey}.

\item We propose a strategy that updates the surrogate adaptively by adopting newly available information during the computation procedure. By doing this, the numerical accuracy of the surrogate model is adaptively improved during the sampling process.
\end{itemize}
To the best of our knowledge, this is the first attempt for adaptive multi-fidelity PC approach for Bayesian inverse problems. Compared to the classical offline approaches where the PC surrogate is built in advance and kept changed during online computations, our strategy is to refine the surrogate model so that a more concentrated region in the parameter space is approximated adaptively. By focusing on the numerical accuracy in a more concentrated region, our multi-fidelity PC surrogate can dramatically enhance the  approximation accuracy, yet with a relatively lower PC expansion order. Thus, the overall computational complexity can be dramatically reduced.

The structure of the paper is as follows. In the next section, we shall review the formulation of Bayesian
inverse problems and the PC-surrogate based approach for inference accelerating. Section 3 contains the main results, we shall propose an adaptive multi-fidelity PC approach to speed up MCMC sampling. In section 4, we use two nonlinear inverse problems to demonstrate the accuracy and efficiency of the proposed method. We finally give some concluding remarks in Section 5.

\section {Background and problem setup}\label{sec:setup}
In this section, we first give a brief overview of the  Bayesian inverse problems. Then we will introduce a surrogate modeling approach via polynomial chaos expansion to accelerate the Bayesian inference.
\subsection{Bayesian inverse problems}

In this paper, we are interested in the problem of  estimating an unknown parameter $z\in \R^{n_z}$ from indirect observations $d\in \R^{n_d}$.  In the Bayesian framework, the prior belief about the parameter $z$ is encoded in the prior probability distribution $\pi(z)$, and  the  distribution of the $ z$  conditioned on the data $d,$ i.e., the posterior distribution $\pi(z|d)$ follows the Bayes' rule,
\begin{eqnarray}\label{ppdf}
\pi(z|d) \propto \pi(d|z) \pi(z),
\end{eqnarray}
where $\pi(d|z)$ is the likelihood function and $\propto$ denotes proportionality up to a scaling constant depends only on $d$ (but not on $z$). For notational convenience, we denote by $\pi^d(z)$ the posterior density $\pi(z|d)$ and by $L(z)$ the likelihood function $\pi(d|z)$. Then, (\ref{ppdf}) can be written as
\begin{eqnarray}\label{ppdf_1}
\pi^d(z)\propto L(z) \pi(z).
\end{eqnarray}
The likelihood $L(z)$ is constructed by the forward model from $Z$ to $d$:
\begin{eqnarray}\label{feq}
d = G(z)+e,
\end{eqnarray}
where the forward model $G$ yields predictions of the data as a function of parameters, and the components of $e$ are i.i.d. random variables that account the observation noises. Now, the likelihood $L(z)$ can be obtained by
\begin{equation}\label{likelihoodfun}
L(z)=\pi_e (d-G(z))=\ds\prod_{i=1}^{n_d}\pi_e(d_i-G_i(z)).
\end{equation}

Notice that if the forward model $G$ is nonlinear, the expression of the likelihood yields a posterior distribution that cannot be written in a closed form.  Standard MCMC methods, e.g. Metropolis-Hastings (MH) sampler, have been extensively used to sample such unknown posterior distributions. However, MCMC requires a large number of repeated evaluations of the likelihood function $L(z)$, and hence relied on repeated evaluations of the forward model $G$, which can be very expensive. Thus, it is of great importance  to construct a relatively cheaper ``surrogate" of the forward model in the offline \cite{Frangos+Marzouk+Willcox2010}. In the next section, we will review the most commonly used polynomial chaos (PC) surrogate approach \cite{ghanem_and_spanos_book,le2010spectral,xiu2010book}.

\subsection{Polynomial chaos expansions}
In order to simplify the discussion and without loss of generality,  in this section,  we describe the PC approximation to the forward model $G$ for $n_d=1$. When $n_d>1$, the procedure will be applied to each component of $G$.  We first assume that the components of the uncertain parameter vector $z=(z^{1},\cdots,z^{n_z})$ are mutually independent  and  $z^{i}$ has marginal probability density $\pi_i(z^{i}): \Gamma_i \rightarrow \mathbb{R}^{+}$. Then
$\pi(z)=\prod^{n_z}_{i=1}\pi_i(z^{i})$
is the joint probability density of the random vector $z$ with the support
$\Gamma:=\prod^{n_z}_{i=1}\Gamma_i \in \mathbb{R}^{n_z}.$

Polynomial chaos expansions represent the model output $G(z)$ as an expansion of orthonormal polynomials of random variables $z$
 \begin{eqnarray} \label{gpcexpansion}
G_N(z)=\sum_{\mb{m} \in \Lambda_N} c_{\mb{m}} \Phi_{\mb{m}}(z),
 \end{eqnarray}
where $\{c_{\mb{m}}\}$ are the unknown expansion coefficients, and the basis functions $\{\Phi_{\mb{m}}\}$ are orthonormal under the density $\pi$, that is,
\begin{eqnarray*}
(\Phi_{\mb{m}},\Phi_{\mb{n}})_{\pi}=\int_{\Gamma}\Phi_{\mb{m}}(z)\Phi_{\mb{n}}(z)\pi(z)dz=\delta_{\mb{m},\mb{n}}.
\end{eqnarray*}
The finite set of multi-indices $\Lambda_N $ is defined as
\begin{eqnarray}\label{mindex}
\Lambda_N :=\{\mb{m}\in \mathbb{N}^{n_z}_0: \|\mb{m}\|_1=\sum^{n_z}_{i=1}|m_i|\leq N\}.
\end{eqnarray}
This specifies a set of polynomials $\{\Phi_{\mb{m}}\}_{\mb{m}\in \Lambda_N}$ such that their total degree $\|\mb{m}\|_1$ is smaller than or equal to a chosen $N$ -- the so called total degree space. The degree of freedom of this space is
\begin{eqnarray}\label{tdterms}
\mbox{card} (\Lambda_N) := M= {n_z+N \choose n_z}.
\end{eqnarray}
 Eq. (\ref{gpcexpansion}) can also be written as
 \begin{eqnarray} \label{gpce}
G_{N}(z)=\sum_{\mb{m} \in \Lambda_N} c_{\mb{m}} \Phi_{\mb{m}}(z) = \sum^M_{m=1} c_m \Phi_m(z).
 \end{eqnarray}
 The above equation implicitly assumes a linear ordering of the elements:
 \begin{eqnarray*}
 \{\Phi_{\mb{m}}\}_{\mb{m} \in \Lambda_N} \Longleftrightarrow \{\Phi_m\}^M_{m=1}.
 \end{eqnarray*}
The main issue in using PC expansion is to efficiently evaluate the unknown coefficients $\{c_m\}$.  In this paper, we use weighted discrete least square method \cite{Narayan2014} to estimate these coefficients. In the discrete least square method (LSM) \cite{Cohen2013stability,Migliorati2013approximation,Tang2014discrete,Zhou2015weighted}, one first chooses a set of nodes $\Theta_{n_z}=\{z_i\}_{i=1}^Q\subset \Gamma$, where $Q\geq M$ is the number of nodes.
The standard LSM seeks to find the PCE coefficients by solving the optimization problem
\begin{eqnarray}
\Big\{c_m\Big\}^M_{m=1} = \arg \min_{c_m} \sum^Q_{i=1}\Big[\Big(G(z_i)-\sum^M_{m=1}c_m\Phi_m(z_i)\Big)\Big]^2.
\end{eqnarray}
This problem can be written algebraically
\begin{eqnarray}\label{lseq}
\mb{c} = \arg \min_{\mb{x} \in \R^M} \|\mb{\Phi x}-\mb{b}\|_2^2,
\end{eqnarray}
where $\mb{c} =(c_1,\cdots, c_M)^T$ denotes the vector of PC coefficients,  $\mb{\Phi}\in \R^{Q\times M}$ denotes the Vandermonde matrix with entries $\mb{\Phi}_{ij}=\Phi_j(z_i), \quad i=1,\cdots, Q, j=1,\cdots, M$, and $\mb{b}=(G(z_1),\cdots,G(z_Q))^T \in \R^Q$ is the vector of samples of $G(z)$.  One can also introduce weights in the least squares formulation. Let $\mb{W} =\mbox{diag} (w_1,\cdots,w_Q)$ be a diagonal matrix with positive entries, a weighted formulation can be written as
\begin{eqnarray}\label{Wlseq}
\mb{c} = \arg \min_{\mb{x} \in \R^M} \|\mb{\sqrt{\mb{W}}\Phi x}-\sqrt{\mb{W}}\mb{b}\|_2^2.
\end{eqnarray}
In this paper, we  consider this formulation with $w_i = \frac{M}{\sum^M_{m=1} \Phi_m^2(z_i)}$ \cite{Narayan2014} to construct the prior-based PC surrogate.

It well-known that the particular sampling strategies for the LSM can dramatically influence the accuracy of the expansion coefficients \cite{Hadigol2018least}. In this work, the samples $z_i, \, i=1,\cdots, Q$ are chosen as i.i.d. samples from a degree-asymptotic density inspired by \cite{Narayan2014}. Specifically, when $\pi(z)$ is uniform measure on $\Gamma=[-1, 1]^{n_z}$ we choose $z_i$ to be sampled from the tensor-product Chebyshev density, and when $\pi(z)$ is the Gaussian measure on $\R^{n_z}$ we choose  the sampling distribution of $z_i$ to have support on the $\R^{n_z}$ unit ball with radius $\sqrt{2N}$. We refer to \cite{Narayan2014} for details on this latter sampling density.

\subsection{PC-surrogate based MCMC sampling}
It is clear that after obtaining the approximation of $\{c_m\}$, one has an explicit functional form $\widetilde{G}_N$. We can then replace the froward model $G$ in (\ref{ppdf_1}) by its approximation $\widetilde{G}_N$, and obtain the surrogate posterior
\begin{eqnarray}\label{ppdf_surrogate}
\widetilde{\pi}^d_N(z)\propto \widetilde{L}_N(z) \pi(z),
\end{eqnarray}
where $\pi(z)$ is again the prior density of $Z$ and $\widetilde{L}_N$ is the approximate likelihood function defined as
\begin{eqnarray}\label{likelihood_su}
\widetilde{L}_N(z):=\ds\prod^{n_d}_{i=1}\pi_{e_i}(d_i-\widetilde{G}_{N,i}(z)),
\end{eqnarray}
where $\widetilde{G}_{N,i}$ is the $i$-th component of $\widetilde{G}_N$. The pseudocode to draw $m$ samples from the approximate posterior distribution $\widetilde{\pi}^d_N(z)$ using  Metropolis-Hastings method is given in Algorithm \ref{alg:MH}. There, a proposal density $q$ is used to draw the next candidate sample. We note that the performance of MCMC depends heavily on the choice of proposal density $q$, and a wide range of proposal densities can be used in Algorithm \ref{alg:MH}. The optimal choice of the proposal density is beyond the scope of this work.

\begin{algorithm}[t]
  \caption{Metropolis-Hastings algorithm using surrogate model.}
  \label{alg:MH}
  \begin{algorithmic}[1]
%   \Procedure{RunChain}{$\widetilde{L}_N,\pi(z),q,m$}
    \Require
      The surrogate  $\widetilde{G}_N$;
      a proposal density $q$, the number of samples $m$; and  a starting points $z_0$;
%    \Ensure
%      Ensemble of classifiers on the current batch, $E_n$;
% \State  Choose a starting points $z_0$
    \For {$i=1,\cdots, m$}
    \State Draw candidate $z^*$ from proposal $q(\cdot|z_{i-1})$, then evaluate the acceptance probability using the surrogate model $\widetilde{G}_N$

%    \State Compute
    \begin{equation*}
    \alpha(z_{i-1}, z^*) = \min \Big\{1, \frac{\widetilde{\pi}^d_N(z^*)q(z_{i-1}|z^*)}{\widetilde{\pi}^d_N(z_{i-1})q(z^*|z_{i-1})}\Big\}
    \end{equation*}
      \If  {Uniform $(0, 1] < \alpha(z_{i-1}, z^*)$ }    \State Accept $z^*$ by setting $z_i =z^*$
                \Else    \State Reject $z^*$ by setting $z_i =z_{i-1}$
                \EndIf
       \EndFor    \State
    \Return $z_1,\cdots, z_m$
%     \EndProcedure
  \end{algorithmic}
\end{algorithm}

The main advantage of the prior-based surrogate method is that, upon obtaining an accurate approximation $\widetilde{G}_N$, whose evaluation is inexpensive compared to the forward model $G$, the approximate posterior density $\widetilde{\pi}^d_N(z)$  can be evaluated for a large number of samples, without resorting to additional simulations of the forward problem.  However, simply replacing the full model $G$ with a surrogate model  $\widetilde{G}_N$  can lead to a bad approximation quality of MCMC solution. As in practice, one can only afford PC expansions with small or moderate PC orders  due to the computational complexity (see \cite{lu2015JCP} and Eq. (\ref{tdterms})). This can obviously introduce a possibly large model error unless the problem is well represented by a low-order PC. If the model error is large, then there might be an dramatic difference between the surrogate posterior and the true posterior \cite{lu2015JCP}. On the other hand, MCMC using the accuracy forward model simulation often leads to computational demands  that exceed available resources. To balance accuracy and efficiency, it is desirable to construct a multi-fidelity model to accelerate the solution of  MCMC, namely, one combining a small number of high-fidelity model evaluations (possibly the true model) and  a much larger number of low-fidelity model evaluations to construct a multi-fidelity surrogate. The key idea of the multi-fidelity approach is that the low-fidelity surrogate model (represented by $u^L$ thereafter) is leveraged for speedup while the true forward model (which is called the high-fidelity model and represented by $u^H$ thereafter) is kept in the MCMC to establish accuracy guarantees \cite{Peherstorfer2016survey}.

\section{Adaptive multi-fidelity polynomial chaos approach}\label{sec:method}
\subsection{Multi-fidelity PC based on LSM}

\begin{algorithm}[t]
  \caption{Multi-fidelity PC based on LSM}
  \label{alg:MPC}
  \begin{algorithmic}[1]
%   \Procedure{RunChain}{$\widetilde{L}_N,\pi(z),q,m$}
    \Require
    The low-fidelity model $u^L=\sum_{\mb{m}\in\Lambda_N} u^L_{\mb{m}} \Phi_{\mb{m}}(z)$;  the high-fidelity model $u^H$;   and  the order of $N_C$;
%    \Ensure
%      Ensemble of classifiers on the current batch, $E_n$;
  \State   Choose Q sampling points $\{z_i\}$ in the parametric space
    \State Calculate the difference between the $u^H(z_i)$ and $u^L(z_i)$
    \State Compute the correction PC coefficients $u^C_{\mb{m}}$ using the least square method
    \State Build the multi-fidelity model by combining $u^L_{\mb{m}}$ and $u^C_{\mb{m}}$ using Eq. (\ref{multieq})
%     \EndProcedure
  \end{algorithmic}
\end{algorithm}

\textcolor{black}{ Recently,  many multi-fidelity approaches have been developed for uncertainty quantification \cite{Amsallem2012nonlinear,Hampton2018practical,Peherstorfer2016multifidelity,Peherstorfer2016survey,Zahr2015progressive}.  In this work, we focus our attention on the multi-fidelity polynomial chaos based on LSM \cite{Ng2012multi,Palar2016multi}.} Suppose that we have a high-fidelity model $u^H$ that estimates the output with very high accuracy (in this work we shall use the true forward model solver as the high-fidelity model $u^H $). We also define a low-fidelity model $u^L$ that estimates the same output with a lower accuracy than the high fidelity model $u^H $. The multi-fidelity approach was originally proposed to enhance the accuracy of a low-fidelity surrogate using supplementary observations of high-fidelity models.  The main idea  is to correct the low-fidelity simulation model using a correction term $C:$
\begin{equation}
C(z) = u^H(z)-u^L(z)\approx \sum_{\mb{m}\in \Lambda_{N_C}} u^C_{\mb{m}} \Phi_{\mb{m}}(z).
\end{equation}

Here the unknown coefficients of the additive correction terms $u^C_{\mb{m}} $ can be calculated by the least squares method. By solving the PC expansions of the correction term, a multi-fidelity model can be approximated via
\begin{equation}
u^M(z)= u^L(z)+C(z)\approx\sum_{\mb{m}\in\Lambda_N} u^L_{\mb{m}} \Phi_{\mb{m}}(z)+\sum_{\mb{m}\in \Lambda_{N_C}} u^C_{\mb{m}} \Phi_{\mb{m}}(z),
\end{equation}
where $u^L_{\mb{m}}$ and $u^C_{\mb{m}} $ are PC coefficients of the low-fidelity and the correction expansions, respectively.

Notice that the construction of $u^C_{\mb{m}}$ requires evaluations of the computationally expensive, the true forward model $u^H$, and therefore the number of this type of evaluations should be much less than the low-fidelity evaluations. Hence the indices of correction expansion must be a subset of low-fidelity expansion indices. For example, consider the total degree polynomial space, to construct an $N$th order multi-fidelity expansion, one can use an $N$-th order low-fidelity expansion combined with an $N_C$-th order ($N_C\leq N$) correction expansion.  The multi-fidelity PC expansion can then be expressed as
\begin{equation}\label{multieq}
u^M(z)=\sum_{\mb{m}\in \Lambda_N} u^L_{\mb{m}} \Phi_{\mb{m}}+\sum_{\mb{m}\in \Lambda_{N_C}} u^C_{\mb{m}} \Phi_{\mb{m}}=\sum_{\mb{m}\in \Lambda_{N_C}} (u^L_{\mb{m}}+u^C_{\mb{m}}) \Phi_{\mb{m}}+\sum_{\mb{m}\in \Lambda_N  \backslash \Lambda_{N_C} } u^L_{\mb{m}} \Phi_{\mb{m}},
\end{equation}
In this way, the multi-fidelity PC introduces an efficient PC approach where the lower-order indices of the low-fidelity PC coefficients are corrected through high-fidelity computations.  When $N_C=N$, all of coefficients will be corrected. The details are shown in Algorithm \ref{alg:MPC}.

Notice that if the multi-fidelity PC $u^M$ is constructed over the whole prior distribution (and kept unchanged), its accuracy for later procedures can generally not be guaranteed. While our main concern in Bayesian inference is the posterior distribution, and thus, it would  suffice to require that $u^M$ is accurate enough only in the posterior density region (while no need to ensure its accuracy everywhere). However, estimation of the high-probability posterior density region is nontrivial as the solution of the inverse problem is unknown until data are available.  Consequently, we shall propose below an adaptive strategy that makes use of immediate data to adaptively build the multi-fidelity surrogate.

\begin{algorithm}[th]
  \caption{Adaptive multi-fidelity PC-based MH algorithm}
  \label{alg:AMPC}
  \begin{algorithmic}[1]
%   \Procedure{RunChain}{$\widetilde{L}_N,\pi(z),q,m$}
    \Require
  Given the subchain length $m$, the maximum allowable number of iterations $I_{max}$,  the upper threshold $\epsilon_0$, the error threshold $\epsilon$, an initial radius $R$ and a constant $\rho$. The prior-based PC surrogate  $\widetilde{G}_N$, the high-fidelity model $u^H=G$ and  a proposal density $q$.
%    \Ensure
%      Ensemble of classifiers on the current batch, $E_n$;
 \State  Choose a starting points $z_0$; let $u^L = \widetilde{G}_N$; let $X_0=\{\}$;
  \For {$n=1,\cdots, I_{max}$}
   \State Draw $m-1$ samples $\{z_1,\cdots,z_{m-1}\}$ from the approximate posterior based on $u^L$ using Algorithm \ref{alg:MH}
   \State Propose $z^*\sim q(\cdot|z_{m-1})$, then evaluation the acceptance probability using high-fidelity model $u^H$
    \begin{equation*}
    \alpha= \min \Big\{1, \frac{\pi^d(z^*)q(z_{m-1}|z^*)}{\pi^d(z_{m-1})q(z^*|z_{m-1})}\Big\}
    \end{equation*}
   \State Set $y=z^*$  with probability $\alpha$, else set $y=z_{m-1}$
%      \If  {Uniform $(0, 1) < \alpha$ }    \State Accept $z^*$ by setting $y =z^*$
%                \Else    \State Reject $z^*$ by setting $y=z_{m-1}$ 
%       \EndIf
      \State Compute the relative error $err(y) $ using Eq. (\ref{conerr})
      \If {$err(y) > \epsilon$}
      \State  Select $Q$ random points $\{x_i\} \in B(y,R)$  and construct the multi-fidelity model $u^M$ using Algorithm \ref{alg:MPC}
      \State Set $u^L =u^M$. If $err(y) \leq \epsilon_0$, set $R=\rho R$
%      \If {$err(y) \leq \epsilon_0$}
 %     \State Set $R=\rho R$
 %     \EndIf
    \EndIf
  % \If {We accept $z^*$ in Line 6}
   \State Evaluation the acceptance probability using  $u^L$
   \begin{equation*}
    \beta= \min \Big\{1, \frac{\widetilde{\pi}^d_N(z^*)q(z_{m-1}|z^*)}{\widetilde{\pi}^d_N(z_{m-1})q(z^*|z_{m-1})}\Big\}
   \end{equation*}
   \If  {Uniform $(0, 1) < \beta$ }    \State Accept $z^*$ by setting $z_m =z^*$
                \Else    \State Reject $z^*$ by setting  $z_m = z_{m-1}$
    \EndIf
  % \EndIf
   \State Let $z_0 =z_m$ and $X_n=X_{n-1} \bigcup \{z_1,\cdots,z_m\}$
   \EndFor
   \State
    \Return Posterior samples $ X_{I_{max}}$
%     \EndProcedure
  \end{algorithmic}
\end{algorithm}

\subsection{Adaptive multi-fidelity PC-based MCMC sampling}
In this section, we shall propose our adaptive sampling framework for constructing the multi-fidelity PC surrogate. Our strategy contains the following steps:
\begin{itemize}
\item  Initialization:  set an initial  surrogate $u^L= \widetilde{G}_N$ where $ \widetilde{G}_N$ is a prior-based PC surrogate.

\item  With the surrogate $u^L$, we can efficiently run MCMC to sample the approximated posterior distribution for a certain number of steps to get the samples $\{z_1,...,z_{m-1}\}$.  Then, the last state $z_{m-1}$ will be used \textcolor{black}{to propose a candidate $z^*$.
\item  If  the surrogate model $u^L$ needs refinement near $z^*$ or $z_{m-1}$, then select new points to construct multi-fidelity model $u^M$ using Algorithm \ref{alg:MPC}.  Set $u^L=u^M$.
\item  Use the surrogate model $u^L$ to accept/reject the proposal  $z^*$.
}
\item The process will be further repeated $I_{max}$ times.  Then the posterior samples can  be generated by gathering all the samples in the above procedures.
\end{itemize}

The detailed algorithm for the above adaptive approach is summarized in Algorithm \ref{alg:AMPC}.  
\textcolor{black}{Choosing when and where to refine the surrogate model $u^L$ is critical to the performance of the adaptive multi-fidelity PC-based MH algorithm. To this end,  in line 3 of Algorithm \ref{alg:AMPC}, we first sample the approximated posterior distribution based on the surrogate model $u^L$ for a certain number of steps using the standard MH algorithm.} The goal is to generate $m-1$ samples so that the initial sample points and the last point are uncorrelated.  In  line 4,  the last sample points $z_{m-1}$ is used to  propose a candidate $z^*$. \textcolor{black}{To decide where to refine the surrogate model, we  compute the  acceptance probability  based on the high-fidelity model $u^H$ (the true forward model).}  In doing this, we can obtain a new accept parameter $y$, which is expected to be much closer to the posterior region than other samples. \textcolor{black}{Next, we decide whether we need to refine the surrogate model at the point $y$.  In line 6, we evaluate the surrogate models and compute  the following absolute error
\begin{equation}\label{conerr}
err(y) = \|u^H(y)-u^L(y)\|_{\infty}.%\frac{\|u^H(z^*)-u^M(z^*)\|_{\infty}}{\|u^H(z^*)\|_{\infty}}.
\end{equation}
Lines 7-10 describe the adaptation of the construction of multi-fidelity surrogate at $y$, which is controlled by the error $err(y)$ and a given threshold $\epsilon$.}  When the  error is less than $\epsilon$, we suppose that the surrogate model is accurate enough and thus it is used directly to accept/reject the proposal. Otherwise, we shall update the current surrogate model.  When refinement of the surrogate model  is needed at the point $y$, we choose $Q$ random points $\{x_i\}$ in a ball centered at $y$,  i.e., $x_i\in B(y,R):=\Big\{x: \|x-y\|_{\infty} \leq R\Big\}$.  The number of points $Q$ depends on the order of the correction expansions $N_C.$ For instance, if $u^M$ is constructed in the total degree space, we shall use $Q = 2 * {N_C+n_z \choose n_z}$ points to compute the expansion coefficients. The choice of the radius $R$ is also a crucial part. Notice that at early iterations, the sample point $y$ might be far away from the high posterior probability region, and thus we could use a relatively large value of $R$. While at later iterations, the refinement procedure will push the samples to approach the high probability density regions, meaning that we should use a relatively small $R.$ This motivates us to introduce a constant $\rho \leq 1$ to control the radius $R$ in line 9. \textcolor{black}{To this end, we introduce an upper threshold $\epsilon_0$. If the error indicator $err$ is smaller than $\epsilon_0$, we let $R=\rho R$.}    In our numerical experiments, the $Q$ random points for each iteration are chosen as $x_i = y+ R \xi_i,$ where $\xi_i \sim U(-1,1).$ \textcolor{black}{To keep the ergodicity of Algorithm \ref{alg:AMPC}, in lines 11-16 we used the surrogate  to accept/reject the proposal $z^*$. } This process will be further repeated $I_{max}$ times.  In this way, a small number of high-fidelity model evaluations are used to update the multi-fidelity PC surrogate, and the evaluations are expected to locate within the high probability region of the inference problem, thus enhance the efficiency whenever the posterior is concentrated.

\textcolor{black}{
\subsection{Theoretical results}
In this section, we analyze the error bound introduced in the exact posterior $\pi^d$ when we use a surrogate model $u^L$ to approximate the high-fidelity model $u^H$. One important result, already proved in \cite{Cui2014data}, given a bound on the Hellinger distance between the two distributions.  In this work, we will focus on bounding  the Kullback-Leibler(KL) divergence which  is defined as
$$D_{KL}(\widetilde{\pi}^d_N||\pi^d):=\int_{\Gamma}\widetilde{\pi}^d_N \log \frac{\widetilde{\pi}^d_N}{\pi^d}dz.$$
}

\textcolor{black}{
We require some notations before stating the result.  First, we define the $\epsilon$-feasible set $\Gamma_N(\epsilon)$ and the associated posterior measure as \cite{Cui2014data} 
\begin{equation}
\Gamma_N(\epsilon) = \Big\{y\in \Gamma \arrowvert \|u^H(y)-u^L(y)\|_{\infty} \leq \epsilon \Big\},
\end{equation}
and
\begin{equation}
\mu\Big(\Gamma_N(\epsilon)\Big)=\int_{\Gamma_N(\epsilon)}\pi^d(z)dz.
\end{equation}
The complement of the $\epsilon$-feasible set is given by $\Gamma^{\perp}_N(\epsilon)=\Gamma \setminus \Gamma_N(\epsilon)$, which has posterior measure $\mu\big(\Gamma^{\perp}_N(\epsilon)\big)=1-\mu\big(\Gamma_N(\epsilon)\big)$.
}

\textcolor{black}{
In this work, the framework set by \cite{Yan+Zhang2017IP} is adapted to analyze the Kullback-Leibler distance between the exact posterior and its approximation induced by the $u^L$ constructed in Algorithm \ref{alg:AMPC}. Following \cite{Yan+Zhang2017IP}, we start with an assumption on the high-fidelity model $u^H$.
\begin{assumption}\label{a1}
The forward operator $u^H:\Gamma \rightarrow \mathbb{R}^{n_d}$ satisfies the following: 
$$\sup_{z\in \Gamma}\|u^H(z)\|:=C_H < \infty.$$
\end{assumption}
}

\textcolor{black}{
Our main convergence result is formalized in the  below.
\begin{theorem}\label{t1}
Assume the functions $u^H$ and $u^L$ satisfy Assumption \ref{a1} uniformly in $N$, and the observational error has an i.i.d. Gaussian distribution. Suppose we have the full posterior distribution $\pi^d$ and its approximation $\tilde{\pi}^d_N$ induced by the surrogate model $u^L$. For a given $\epsilon> 0$, there exist constants $K_1 >0$ and $K_2 >0$ such that
\begin{equation}\label{KLbound}
D_{KL}(\tilde{\pi}^d_N||\pi^d)\leq \Big(K_1 \epsilon+K_2\mu\big(\Gamma^{\perp}_N(\epsilon)\big)\Big)^2.
\end{equation}
\end{theorem}
\begin{proof}
See  Appendix for the detailed proof.
\end{proof}
}

\textcolor{black}{
Notice that if the approximation is accurate, the expression in (\ref{KLbound}) will be small. Specifically, if the posterior distribution is sufficiently well sampled such that 
\begin{equation}\label{cond1}
\mu\Big(\Gamma^{\perp}_N(\epsilon)\Big)\leq \epsilon,
\end{equation}
 then the KL distance is characterized entirely by $\epsilon^2$. Thus, by adaptively updating the multi-fidelity model until condition (\ref{cond1}) is satisfied, we can build an approximate posterior distribution whose error, measured by Kullback-Leibler divergence, can be bounded in terms of  $\epsilon^2$, as shown in Theorem \ref{t1}.   In Algorithm \ref{alg:AMPC}, if a  candidate  point $y\in \Gamma^{\perp}_N(\epsilon)$,  we will select new points at a ball $B(y, R)$ to construct multi-fidelity model using Algorithm \ref{alg:MPC}.  In this way, the posterior measure of $\mu\big(\Gamma^{\perp}_N(\epsilon)\big)$ decays asymptotically with refinement of the surrogate model, and  the adaptive construction process can search the parameter space more thoroughly to increase the likelihood that the condition (\ref{cond1}) is satisfied. %Intuitively,  the $\epsilon$-feasible set asymptotically grows with refinement of the multi-fidelity model, and hence, $\mu\Big(\Gamma^{\perp}_N(\epsilon)\Big)$ asymptotically decays.  
 }
 
 \textcolor{black}{
 According to the relation between the  Hellinger distance and Kullback-Leibler distance, Theorem \ref{t1} can provide an error bound between $\tilde{\pi}^d_N$ and $\pi^d$ in the Hellinger distance. 
 \begin{corollary}[\cite{Cui2014data}]\label{t2}
Assume the functions $u^H$ and $u^L$ satisfy Assumption \ref{a1}, and the observational error has an i.i.d. Gaussian distribution. Suppose we have the full posterior distribution $\pi^d$ and its approximation $\tilde{\pi}^d_N$ induced by the surrogate model $u^L$. For a given $\epsilon> 0$, there exist constants $K_1 >0$ and $K_2 >0$ such that
\begin{equation*}
D_{Hell}(\tilde{\pi}^d_N||\pi^d)\leq K_1 \epsilon+K_2\mu\big(\Gamma^{\perp}_N(\epsilon)\big).
\end{equation*}
\end{corollary}
 }

\section{Numerical Examples}\label{sec:tests}

In this section, we present two practical PDE inverse problems to illustrate the accuracy and efficiency of the adaptive multi-fidelity PC approach. The first example is adapted from \cite{Li+Marzouk2014SISC,Marzouk+Najm+Rahn2007,yan+guo2015}, where a 2D heat source inversion problem is considered. The parametric dimension of this problem is set as $n_z=2$, and this allows us to use higher polynomial orders to construct an accurate prior-based PC surrogate. Although in this case, our AMPC approach does not significantly reduce the computation time (compared to the prior-based PC method), it does offer significant improvement in  accuracy. The second example is the inferring of the diffusion coefficient in an elliptic PDE with nine random parameters \cite{Cui2014data}. This is a more challenging problem, since constructing a global accurate prior-based PC surrogate is very expensive due to the high dimension of the parametric space. We shall use this example to show that our approach can dramatically enhance both the efficiency and the accuracy, compared to the prior-based PC method.

In all our tests, unless otherwise specified, we shall use the following parameters  $N_C= 2, R=0.1, \rho=0.5, \epsilon = 1\times 10^{-3}, \epsilon_0 =  0.1$.   We shall run the AMPC approach using Algorithm \ref{alg:AMPC} for $I_{max} =10$ iterations, with a subchain length $m=5,000$.  We shall  provide comparisons between AMPC and the prior-based PC approach \cite{Yan+Zhang2017IP}. To make a fair comparison, we shall run the prior-based PC approach for $50,000$ iterations. Moreover, the MCMC simulation results using the high-fidelity model (referred as the conventional MCMC approach) will also be conducted, and this is used as the reference solution to evaluate the accuracy and efficiency of the two methods (AMPC and the prior-based PC approach). For all algorithms, the same Gaussian proposal distribution will be used.  To present the numerical results, the last $30,000$ realizations will be used to compute the relevant statistical quantities. All the computations are performed using MATLAB 2015a on an Intel-i5 desktop computer.

\subsection{Example 1: 2D heat source inversion}
Consider the following model in the physical domain $D=[0,1]\times [0,1]$
\begin{eqnarray}\label{2dsource}
\begin{array}{rl}
^cD_t^{\alpha}u-\nabla ^2 u&=e^{-t}\exp\Big[-0.5\Big(\frac{\|Z-\mathbf{x}\|}{0.1}\Big)^2\Big],\quad D\times [0, 1],\\
\nabla u \cdot \textbf{n}&=0, \quad \mbox {on} \,\,\, \partial{D},\\
u(\textbf{x},0)&=0, \quad  \mbox{in}\,\,\, D.
 \end{array}
\end{eqnarray}
The goal is to determine the source location $Z=(Z_1,Z_2)$ from noisy measurements of the $u$-field at a finite set of locations and times.  Here $^c D^{\alpha}_t \,(0<\alpha<1)$ denotes the Caputo fractional derivative of order $\alpha$ with respect to $t$  and it is defined by \cite{Podlubny1999}
\begin{eqnarray*}
^cD_t^{\alpha}u(\mathbf{x},t)=\frac{1}{\Gamma(1-\alpha)} \int ^t_0 \frac{\partial{u(\mathbf{x},\eta)}}{\partial{\eta}} \frac{d\eta}{(t-\eta)^{\alpha}}, \,\,\, 0<\alpha<1,
\end{eqnarray*}
where $\Gamma (\cdot)$ is the Gamma function. Here the Caputo fractional derivative of order $\alpha$ is chosen as 0.8.

In the numerical simulation, we shall use a finite difference/ spectral approximations (\cite{Lin+Xu2007}) with time step $\Delta t=0.01$ and polynomial degree $P=6$ to solve the equation (\ref{2dsource}).  Then, the simulation data are generated by adding independent random noise $N(0,\sigma^2)$ to the deterministic simulation results at a uniform $3 \times 3$ sensor network. At each sensor location, two measurements are taken at time $t = 0.25$ and $t = 0.75$, which corresponds to a total of  18 measurements.  The prior on $Z$ reflects a uniform probability assignment over the entire domain of possible source locations, i.e., $Z_i\sim U(0,1)$.   Note that the standard deviation  $\sigma$ may be difficult to quantify directly, especially when the experiment for data acquisition is not repetitive. In this section,  the noise level $\sigma$ is assumed to be unknown and an inverse Gamma distribution with parameters $(1\times 10^{-3}, 1\times 10^{-3})$ is assumed  for $\sigma^2$.

In order not to commit an `inverse crime', we generate the data by solving the forward problem at a much finer mesh than that is used in the inversion ($P=10$).

%%%%%%%%%%  Fig
\begin{figure}
\begin{center}
  \begin{overpic}[width=0.45\textwidth,trim=20 0 20 15, clip=true,tics=10]{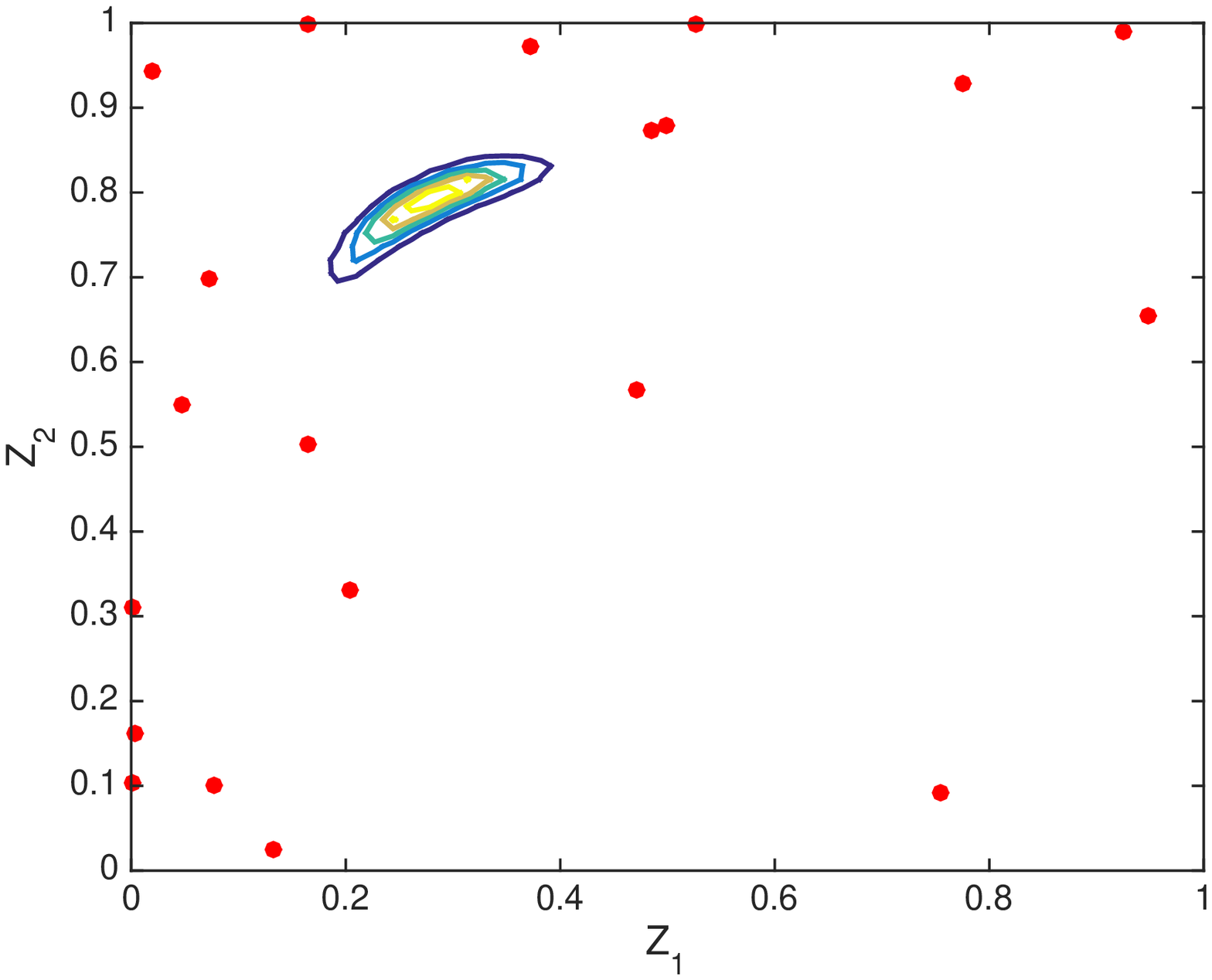}
  \end{overpic}
    \begin{overpic}[width=0.45\textwidth,trim= 20 0 20 15, clip=true,tics=10]{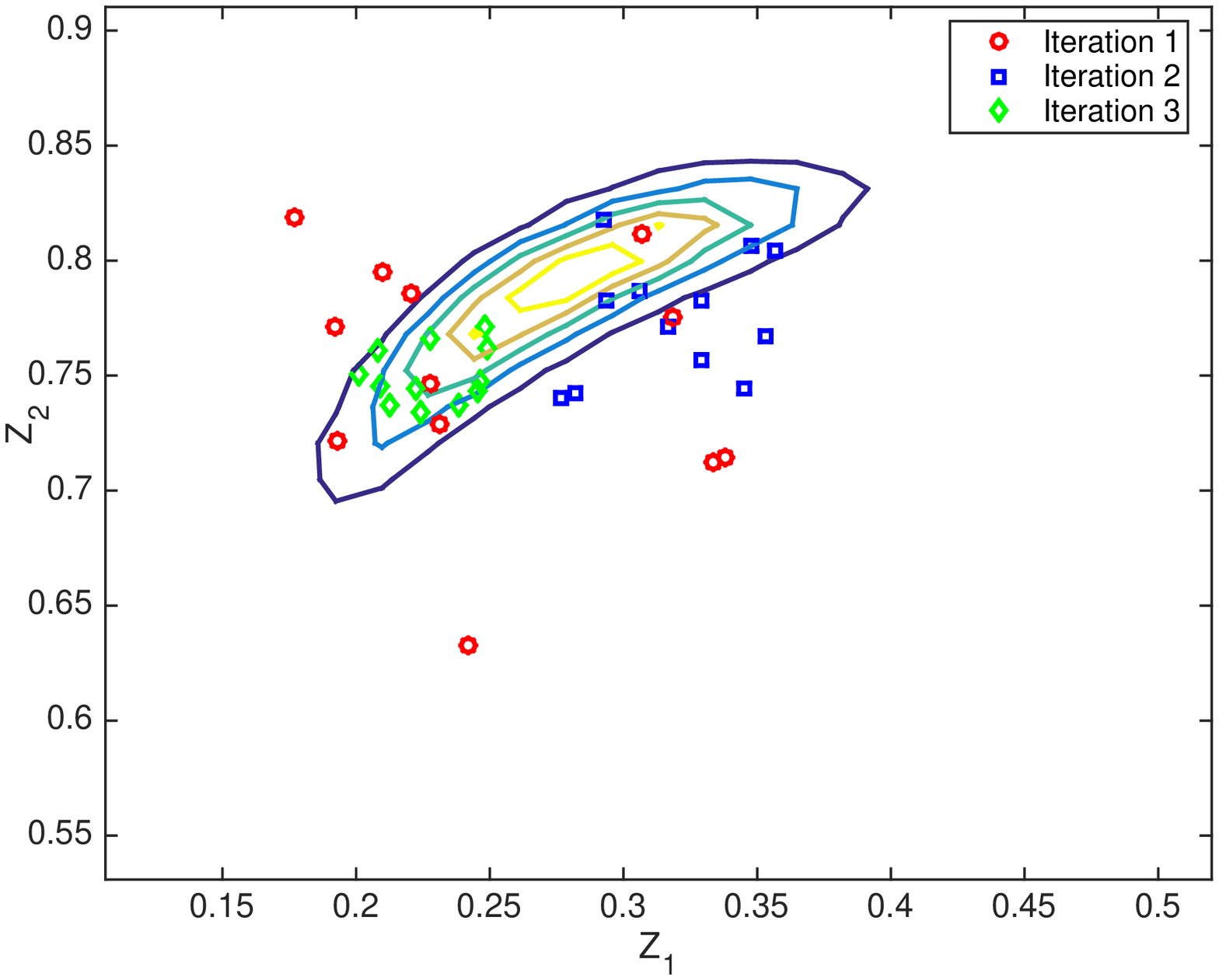}
  \end{overpic}
\end{center}
\caption{Model evaluation points used to construct the prior-based PC surrogates (Left) and AMPC (Right) surrogates.}\label{points_eg1}
  \end{figure}

\begin{figure}
\begin{center}
  \begin{overpic}[width=\textwidth,trim=20 0 20 15, clip=true,tics=10]{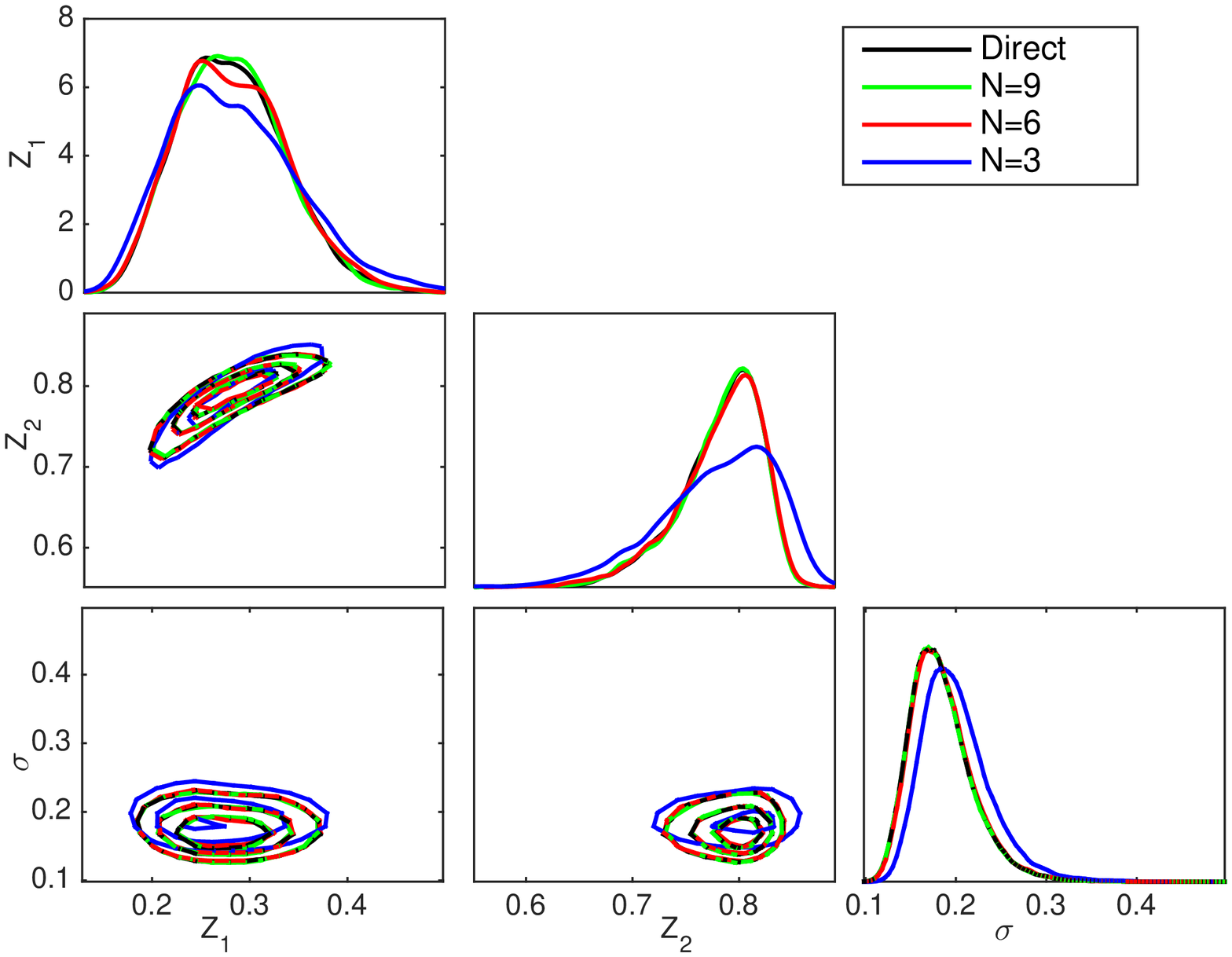}
  \end{overpic}
\end{center}
\caption{"True" posterior density (Black line), compared with the posterior density obtained via the prior-based PC surrogate. $\sigma =0.2$.}\label{pro_pos_s20}
\end{figure}

\begin{figure}
\begin{center}
  \begin{overpic}[width=\textwidth,trim=20 0 20 15, clip=true,tics=10]{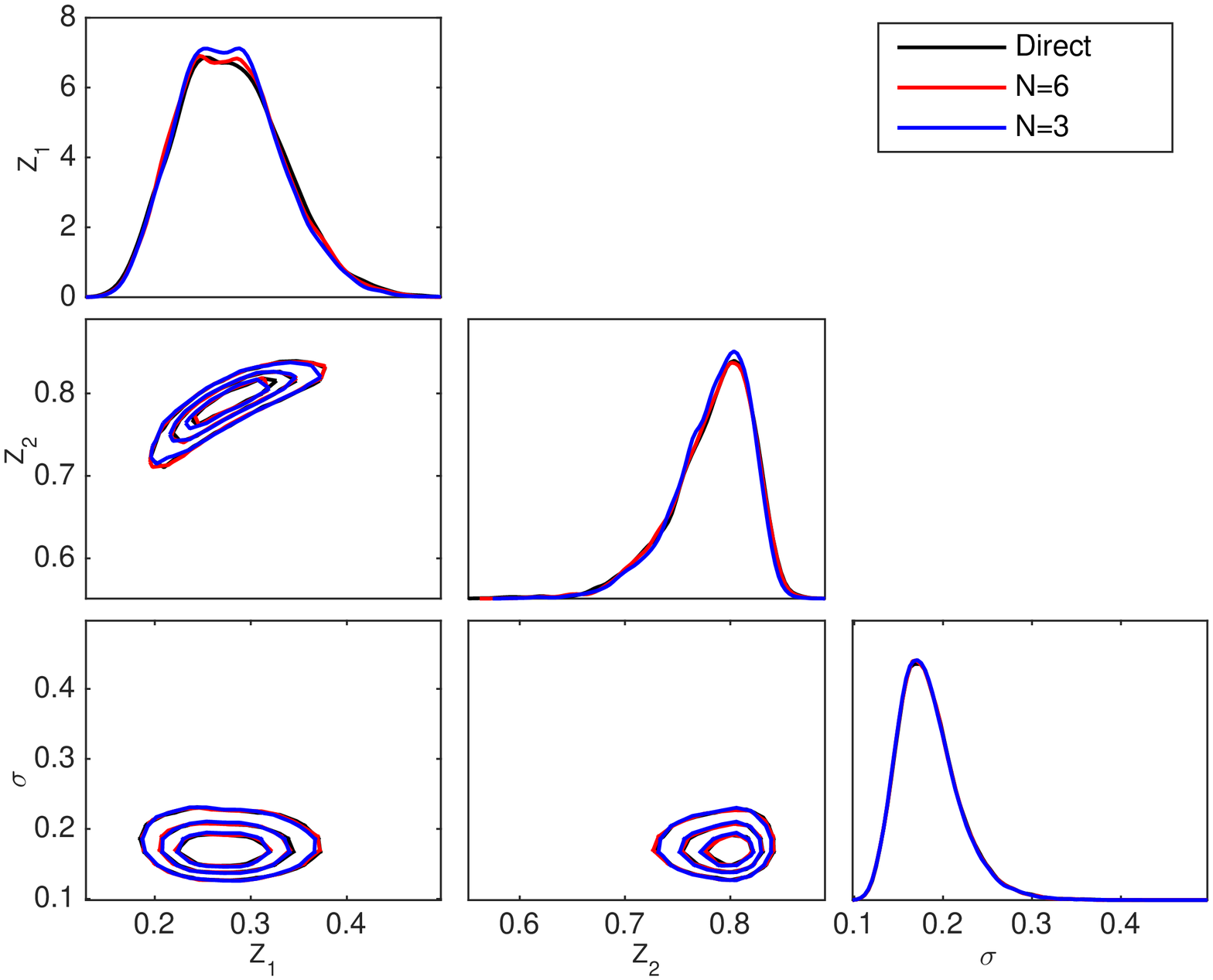}
  \end{overpic}
\end{center}
\caption{"True" posterior density (Black line), compared with the posterior density obtained via the AMPC surrogate. $\sigma =0.2$.}\label{adap_pos_s20}
\end{figure}

 \begin{table}[tp]
      \caption{Example 1 ($\alpha$ is known). Computational times, in seconds, given by three different methods. $\epsilon=1\times10^{-3}, \sigma=0.2$. }\label{eg1_time}
  \centering
  \fontsize{6}{12}\selectfont
  \begin{threeparttable}
    \begin{tabular}{ c ccccc}
  \toprule
 & \multicolumn{2}{c}{Offline}&\multicolumn{2}{c}{Online}\cr
\cmidrule(lr){2-3} \cmidrule(lr){4-5}

  \multirow{1}{*}{Method}  &$\text{$\#$ of model evaluations}$&CPU(s) &$\text{$\#$ of model evaluations}$&CPU(s)     &\multirow{1}{*}{Total time(s)}\cr
  \midrule
    Direct                                     & $-$       & $-$         & 5$\times 10^4$    &$\text{\bf{4424.6}}$      & $\text{\bf{4424.6}}$   \cr
   PC, $N=9$                & 110  & 8.1   & $-$                       &$\text{\bf{17.3}}$      & $\text{\bf{25.4}}$   \cr
   PC, $N=6$                & 56   & 4.1   & $-$                       &12.2                          & 16.3  \cr
   PC, $N=3$                & 20   & 1.5   & $-$                        & 7.7                          & 9.2\cr
  AMPC, $N=6, N_C=2$        & 56       & 4.1        & 32                          & $\text{\bf{14.5}}$       & $\text{\bf{18.6}}$   \cr
  AMPC, $N=3, N_C=2$        & 20       & 1.5         & 56                   & $\text{\bf{12.1}}$       & $\text{\bf{13.6}}$   \cr
    \bottomrule
      \end{tabular}
    \end{threeparttable}

\end{table}

\begin{figure}
\begin{center}
  \begin{overpic}[width=\textwidth,trim=20 0 20 15, clip=true,tics=10]{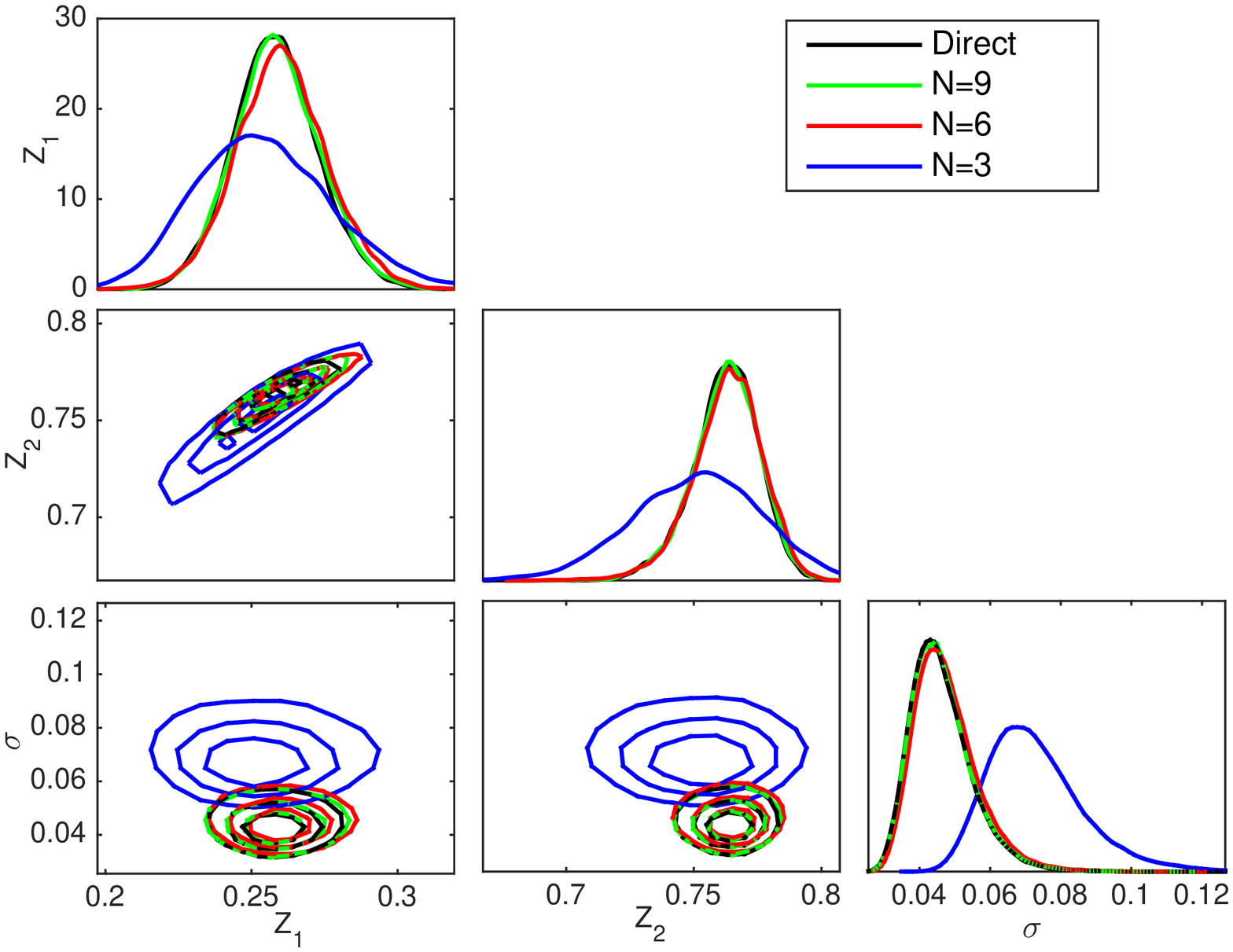}
  \end{overpic}
\end{center}
\caption{"True" posterior density (Black line), compared with the posterior density obtained via the prior-based PC surrogate. $\sigma =0.05$.}\label{pro_pos_s05}
\end{figure}

\begin{figure}
\begin{center}
  \begin{overpic}[width=\textwidth,trim=20 0 20 15, clip=true,tics=10]{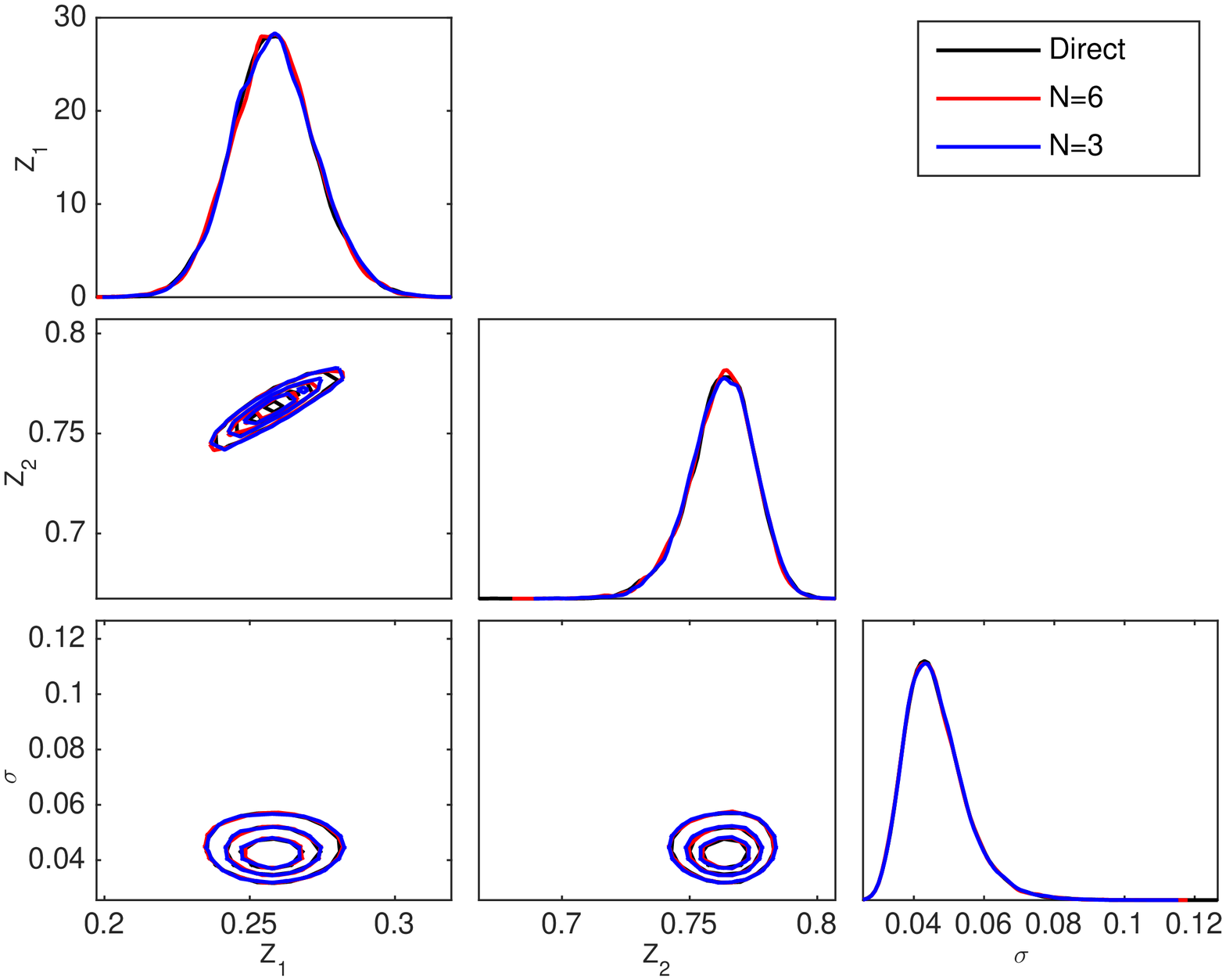}
  \end{overpic}
\end{center}
\caption{"True" posterior density (Black line), compared with the posterior density obtained via the AMPC surrogate. $\sigma =0.05$.}\label{adap_pos_s05}

\bigskip
%\begin{figure}
\begin{center}
    \begin{overpic}[width=0.45\textwidth,trim= 20 0 20 15, clip=true,tics=10]{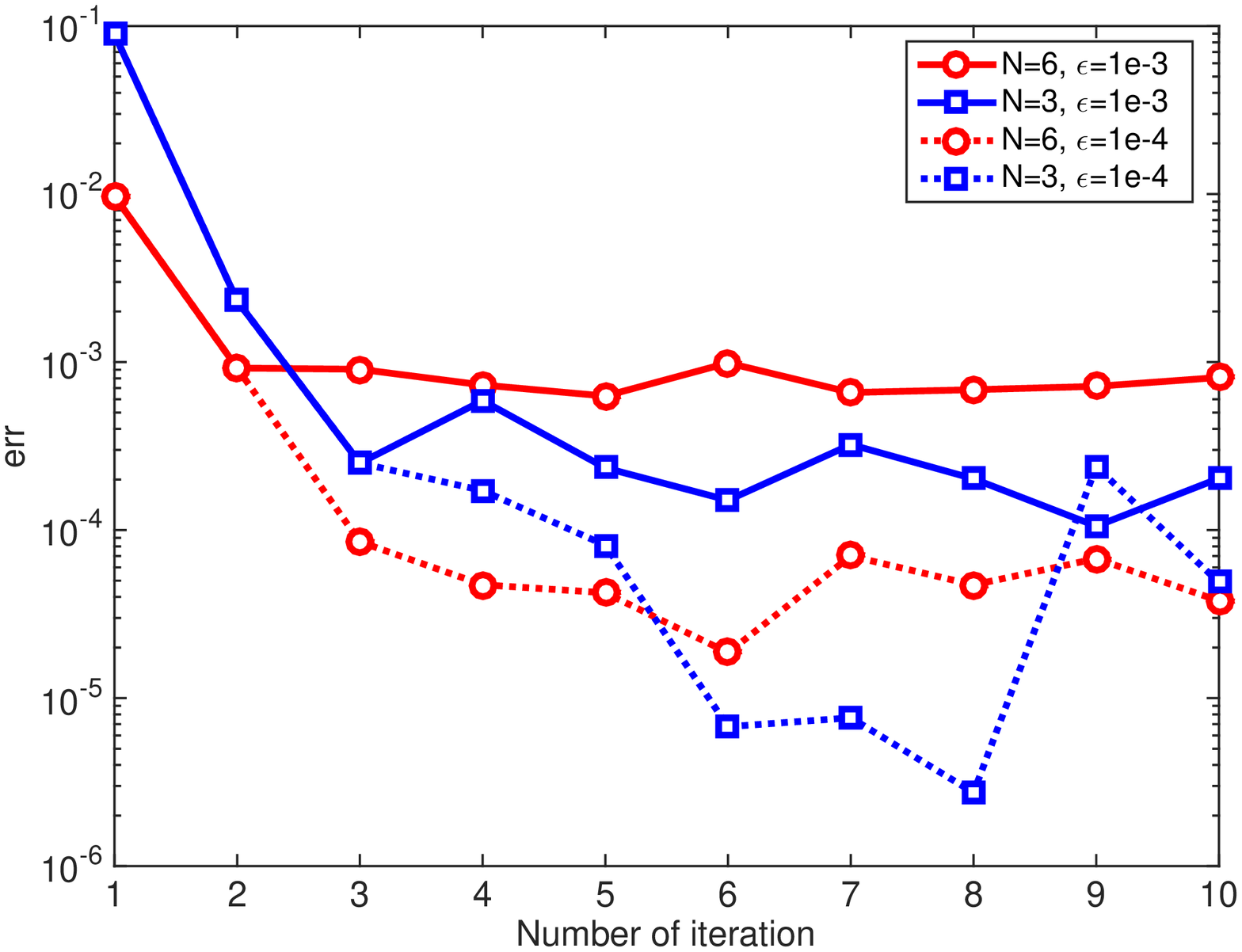}
  \end{overpic}
\end{center}
\caption{Example 1: the relative errors with number of iteration for AMPC approach.}\label{error-f1}
\end{figure}
%%%%%%%%%%%%%%%%%

\subsubsection{Determine the source location $Z$ when $\alpha$ is known}\label{sec:1}

With this test, we shall show that the AMPC approach can improve the accuracy of the prior-based PC approach without a significantly increase in computational time. We set the exact source location as $(0.25, 0.75)$ and set the standard deviation as $\sigma =0.2$. To better present the results, we shall perform the following three-types of approaches:
\begin{itemize}
\item The \textit{conventional} MCMC, or the direct MCMC approach based the forward model evaluations. This approach is supposed to be time consuming, yet with high accuracy. The associated solution can be viewed as a reference solution.
\item The MCMC approach with a prior-based PC surrogate --  \textit{the prior-based PC approach} for short.
\item The AMPC approach presented in Section 3.
\end{itemize}

In our computations, the initial guess for $Z$ and $\sigma$ are chosen as $(0, 0)$ and $1$, respectively.  For the AMPC surrogate,  we first construct a prior-based PC surrogate $\widetilde{G}_{N}$ with $N=3$ using 20 model evaluations.  Then we set the initial multi-fidelity surrogate as $u^M_0=\widetilde{G}_{N}$.  Using $u^M_0$ and the true forward model, we can draw a random sample $y$ via Algorithm \ref{alg:AMPC}, which is expected to be much closer to the posterior region. When the error indicator $err$ (\ref{conerr}) exceeds the threshold $\epsilon=1\times 10^{-3}$, we choose $Q=2 {2+2 \choose 2} =12$ random points in the region $B(y,R)$ to refine the multi-fidelity PC surrogate $u_0^M$. Then this procedure is further repeated 10 times (or until the threshold is reached). We first show how the locations of the samples (the points in parameter space at which the high-fidelity model are evaluated ) evolve in our adaptive procedure. As shown in Figure \ref{points_eg1} (Left), although 20 model evaluation points (red color) are used in the prior-based PC method, none of them actually fall in the region of significant posterior probability. This is one of the main drawbacks of the prior-based PC surrogate. In contrast, with the adaptive multi-fidelity approach (Figure \ref{points_eg1}, right), \textcolor{black}{4 of the 12} model evaluations (red points) occur in the important region of the posterior distribution for the first iteration. In the third iteration,  almost all of the evaluation points (green points) fall into the important region of the posterior distribution.

Next we investigate the convergence behavior of the prior-based PC approach and our AMPC approach. Notice that the accuracy and computational cost of the prior-based PC method depend on the order of the PC expansion. We have shown in Figure \ref{pro_pos_s20} the one and two dimension marginal posterior distributions of the three parameters by the prior-based PC approach with varying $N$. Very close agreement with the true solution is observed with $N=6$, and this agreement improves further with $N=9$.  However, a lower order PC expansion ($N$=3) results in a very poor density estimate. The corresponding results obtained by AMPC are shown in Figure \ref{adap_pos_s20}.  The black-solid lines are the marginal posterior probability densities estimated by the conventional MCMC (the reference solution), and the red and blue-solid lines represent the approximations by the AMPC approach with $N=6$ and $N=3$, respectively. It is clearly shown that the AMPC approach results in a good approximation to the reference solution. Even with a lower order PC expansion $N=3$, the posterior density obtained by AMPC agrees well with the reference solution. By comparing Figures \ref{pro_pos_s20} and \ref{adap_pos_s20}, we learn that, with the same order of PC expansion, the approximation results using AMPC are much more accurate than those using the prior-based PC method.

The computational times, in seconds, given by three different approaches are presented in Table \ref{eg1_time}. The main computational time in the prior-based PC approach is the offline high-fidelity model evaluations. The number of such model evaluations for with $N=\{9, 6, 3\}$  are 110, 56 and 20,  respectively. Upon obtaining the prior-based PC surrogate, the online MCMC simulation  is very cheap as it does not require any high-fidelity model (the true forward model) evaluations. For the AMPC, however, we indeed need the online high-fidelity model simulations. Nevertheless, in contrast to $5\times 10^4$ high-fidelity model evaluations in the conventional MCMC, the number of high-fidelity model evaluations for the AMPC approach with $N=\{6,3\}$ are only 32 and 56,  respectively.  As can be seen from the fifth column of Table \ref{eg1_time}, the AMPC approach does not significantly increase the computation time compared to the prior-based PC approach, yet it offers significant improvement in the accuracy. With the same accuracy, the AMPC approach actually uses much less high-fidelity model evaluations than the prior-based PC approach (with $N=9$) and the conventional MCMC. This confirms the efficiency of the AMPC approach.

It is well-known that the smaller the standard deviation in the noise distributions is, the more accurate the prior-based PC surrogate should be constructed, otherwise, the error between the true posterior and its approximation reduced by the surrogate model could be very large \cite{Yan+Zhang2017IP}.  To investigate the influence of the variance of the measurement error, we run the example  with $\sigma = { 0.05}$. The corresponding results are shown in Figures \ref{pro_pos_s05}-\ref{adap_pos_s05}. Compare with Figure \ref{pro_pos_s20}, it is expected to see that the prior-based PC approach with $N=3$ yields significantly worse results for $\sigma = { 0.05}$.  In contrast, the approximated posterior distributions computed with the AMPC approach with $N=3$ and $6$ are shown in Figure \ref{adap_pos_s05}. We can see that the AMPC approach can provide much improved results in this situation.

We finally plot the relative errors $err$ with the number of iteration in Figure \ref{error-f1}.  The effect of the threshold $\epsilon$ is also investigated in this figure. For $\epsilon=1\times 10^{-3},$ it is shown that the relative error decays with respect to the number of iterations. For $\epsilon=1\times 10^{-4},$ we notice that the error decreases exponentially up to the third iteration and then oscillates slightly around $5\times 10^{-5}$.  Note that the refinement occurs whenever $err$ exceed the threshold $\epsilon$. It is observed that a lower $\epsilon$ can result in higher accuracy approach, however, the computational cost is also enhanced due to more high-fidelity model evaluations.

 \subsubsection{Determine the source location $Z$ when $\alpha$ is unknown}

 \begin{figure}
\begin{center}
  \begin{overpic}[width=0.45\textwidth,trim=20 0 20 15, clip=true,tics=10]{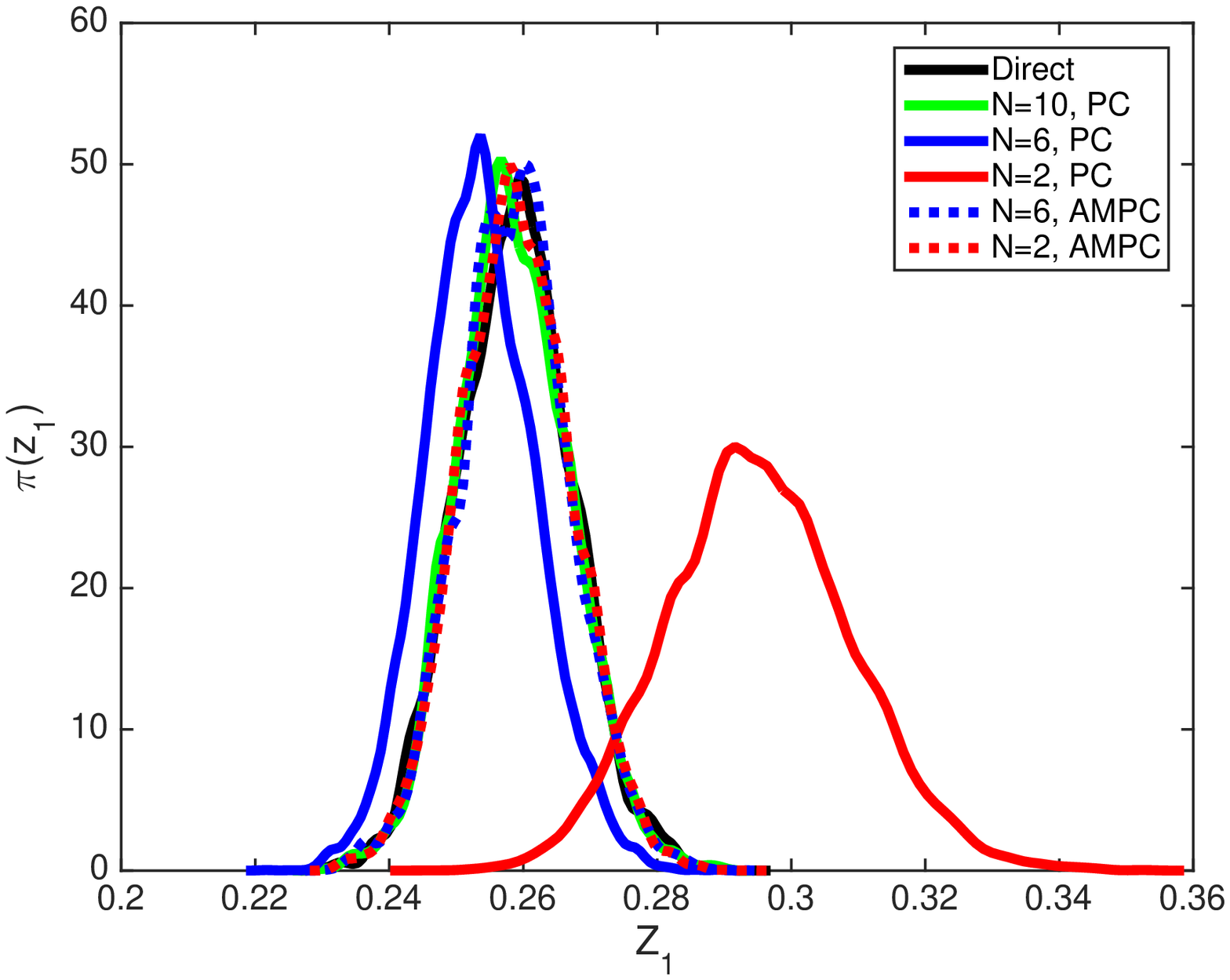}
  \end{overpic}
    \begin{overpic}[width=0.45\textwidth,trim= 20 0 20 15, clip=true,tics=10]{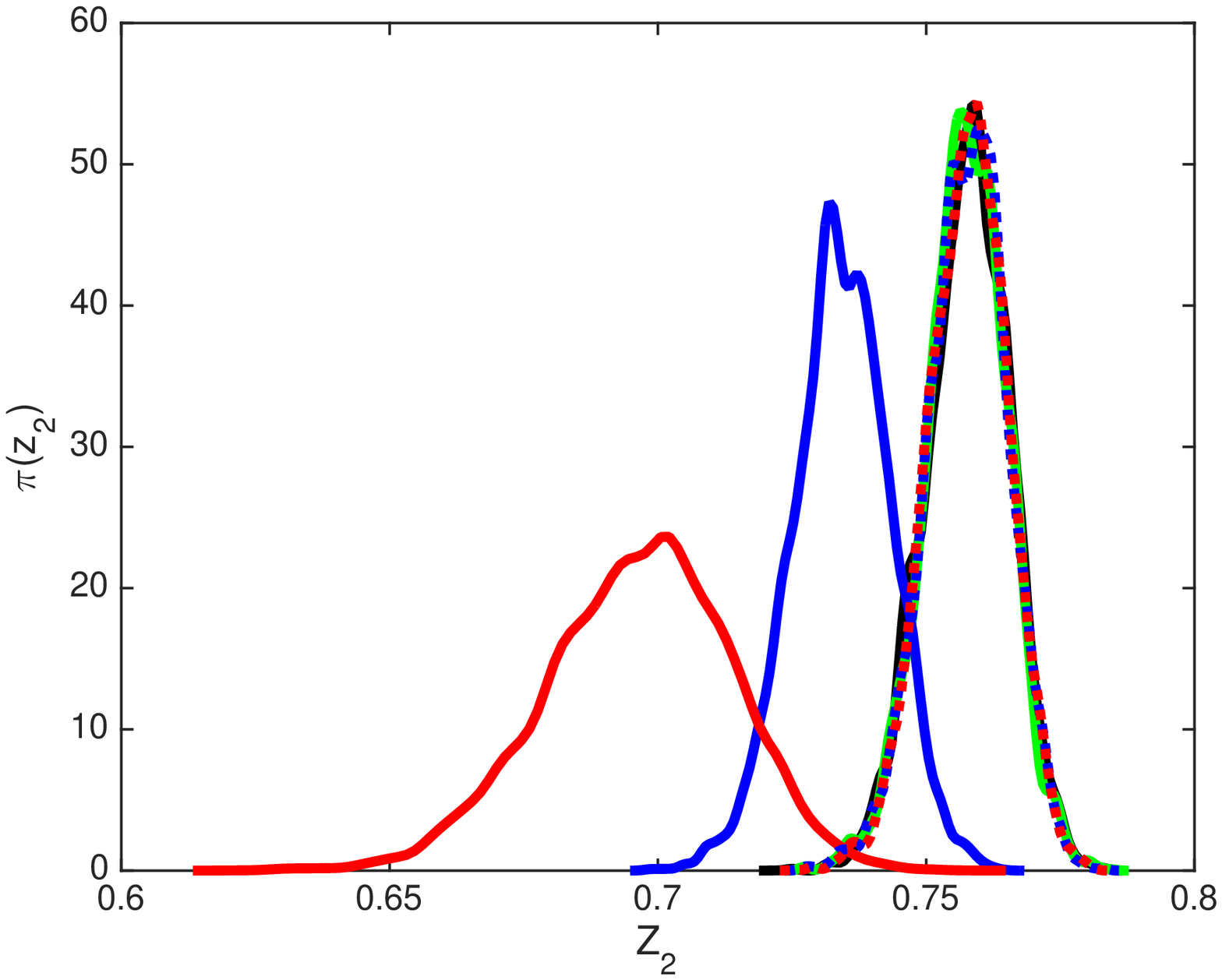}
  \end{overpic}
    \begin{overpic}[width=0.45\textwidth,trim= 20 0 20 15, clip=true,tics=10]{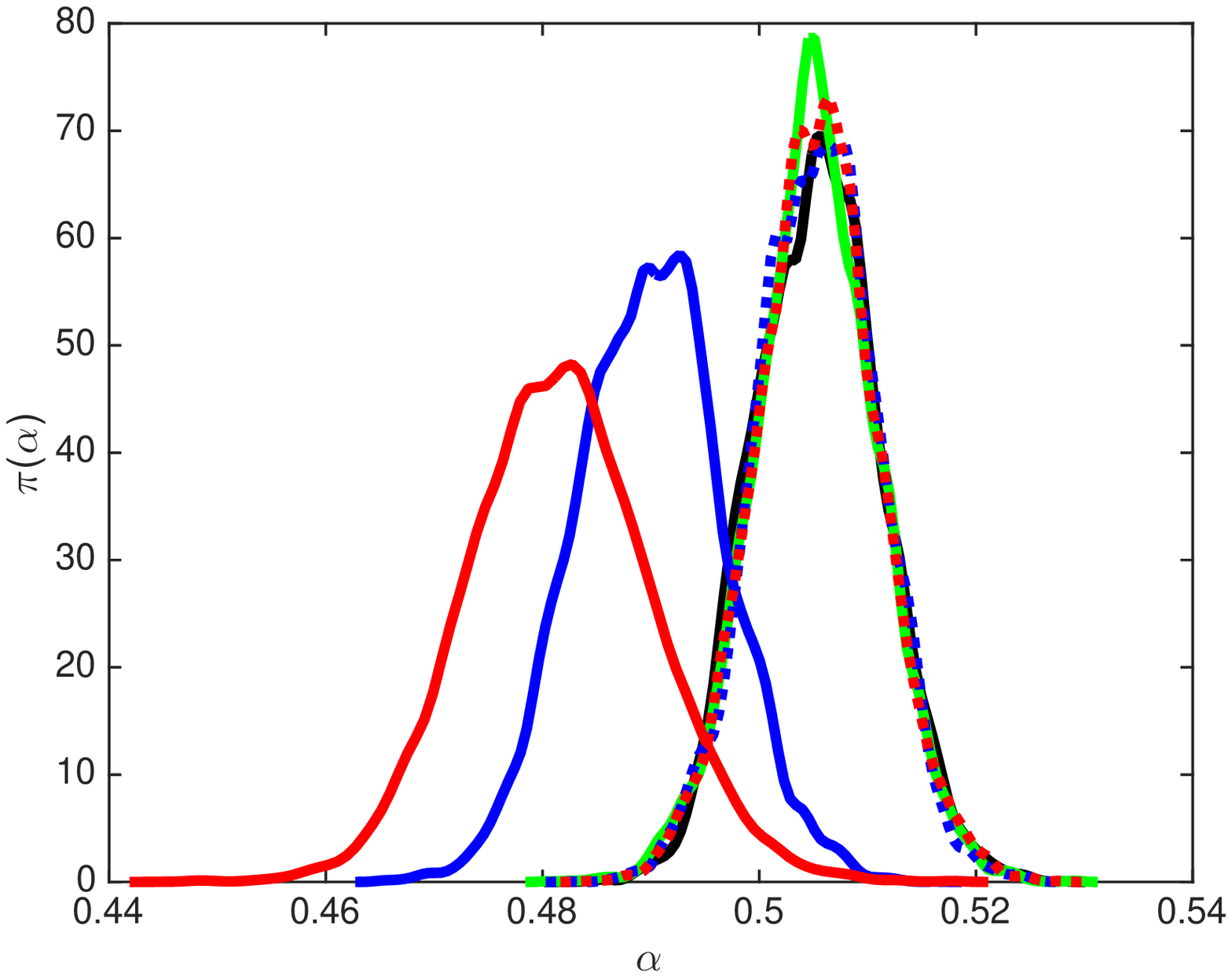}
  \end{overpic}
    \begin{overpic}[width=0.45\textwidth,trim= 20 0 20 15, clip=true,tics=10]{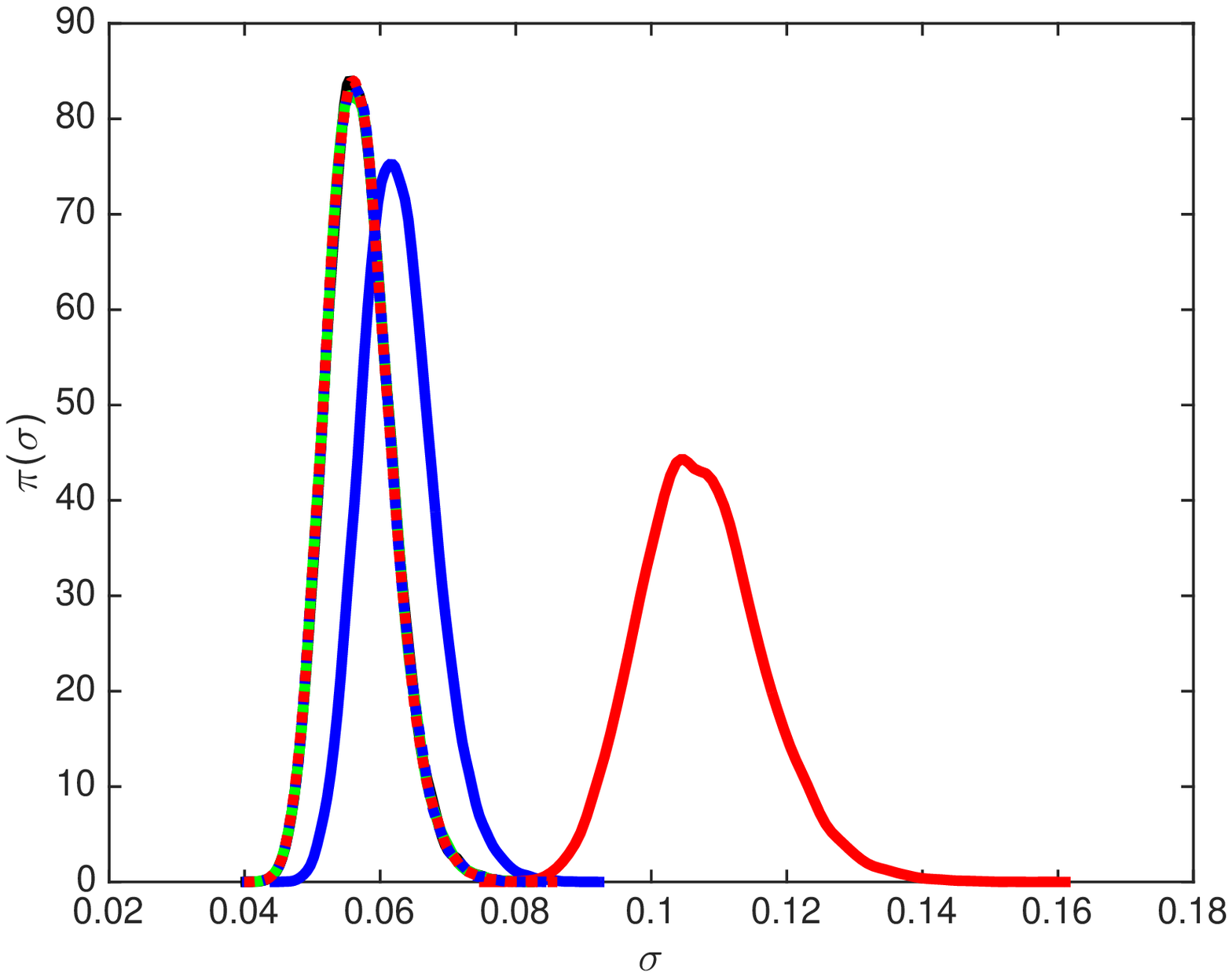}
  \end{overpic}
\end{center}
\caption{Example 1 ($\alpha$ is unknown): The marginal distribution of each component of the parameter. Densities result from prior based PC surrogate and AMPC approach; compared to direct method.}\label{pi-eg2}
\end{figure}

 \begin{table}[tp]
      \caption{Example 1 ($\alpha$ is unknown). Computational times, in seconds, given by three different methods. $\epsilon=1\times10^{-3}, \sigma=0.05$. }\label{eg1_time_2}
  \centering
  \fontsize{6}{12}\selectfont
  \begin{threeparttable}
    \begin{tabular}{ c ccccc}
  \toprule
 & \multicolumn{2}{c}{Offline}&\multicolumn{2}{c}{Online}\cr
\cmidrule(lr){2-3} \cmidrule(lr){4-5}

  \multirow{1}{*}{Method}  &$\text{$\#$ of model evaluations}$&CPU(s) &$\text{$\#$ of model evaluations}$&CPU(s)     &\multirow{1}{*}{Total time(s)}\cr
  \midrule
    Direct                                     & $-$       & $-$         & 5$\times 10^4$    &$\text{\bf{4483.7}}$      & $\text{\bf{4483.7}}$   \cr
   PC, $N=10$                & 572  & 41.5   & $-$                       &$\text{\bf{26.7}}$      & $\text{\bf{68.2}}$   \cr
   PC, $N=6$                  & 168   & 12.3   & $-$                       &16.1                          & 28.4  \cr
   PC, $N=2$                  & 20     & 1.5     & $-$                        & 6.9                          & 8.4 \cr
  AMPC, $N=6, N_C=2$        & 168       & 12.3          & 60                         & $\text{\bf{21.2}}$       & $\text{\bf{33.5}}$   \cr
  AMPC, $N=2, N_C=2$        & 20       & 1.5            & 80                        & $\text{\bf{13.5}}$       & $\text{\bf{15.0}}$   \cr
    \bottomrule
      \end{tabular}
    \end{threeparttable}

\end{table}

  \begin{figure}
\begin{center}
  \begin{overpic}[width=0.45\textwidth,trim=20 0 20 15, clip=true,tics=10]{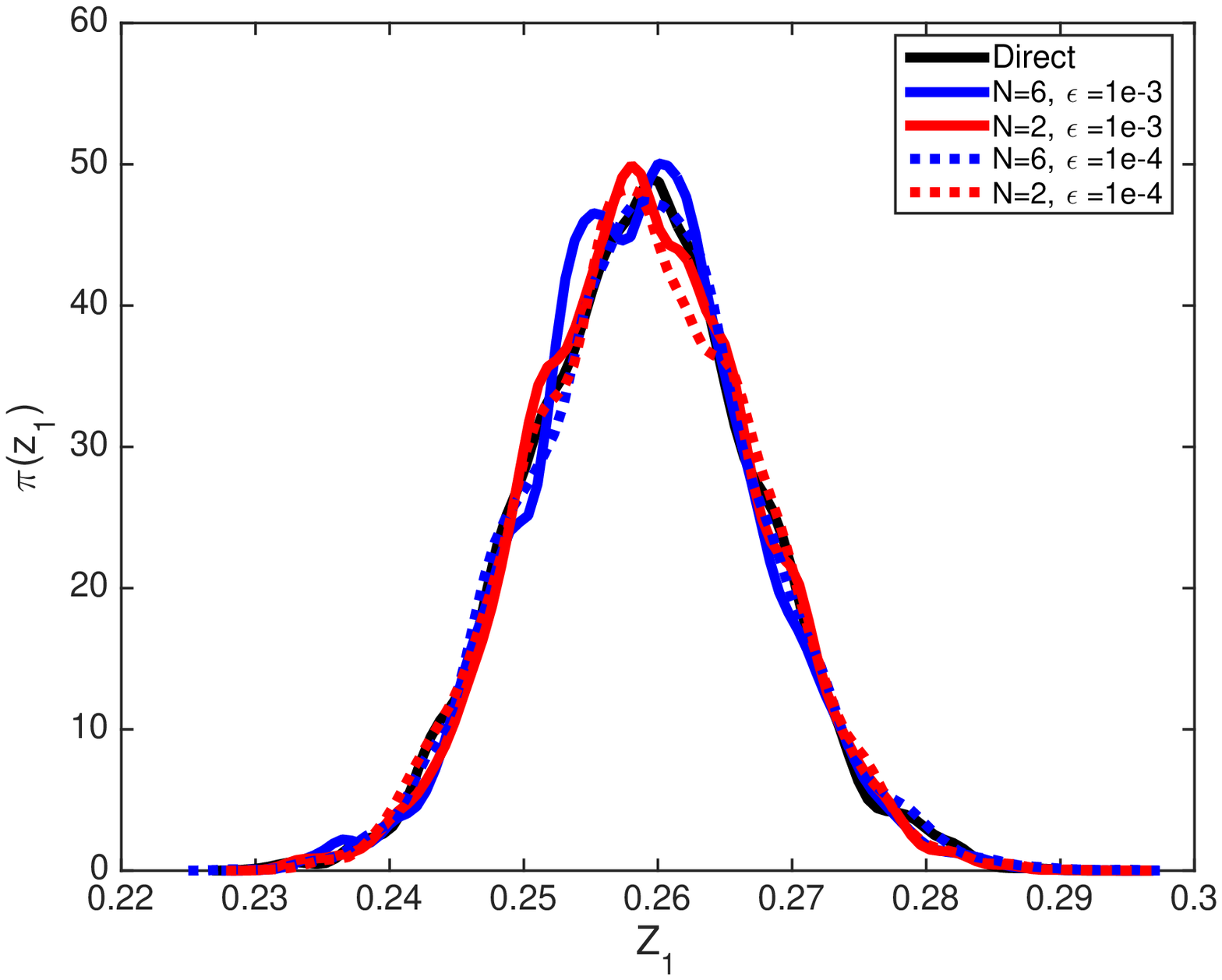}
  \end{overpic}
    \begin{overpic}[width=0.45\textwidth,trim= 20 0 20 15, clip=true,tics=10]{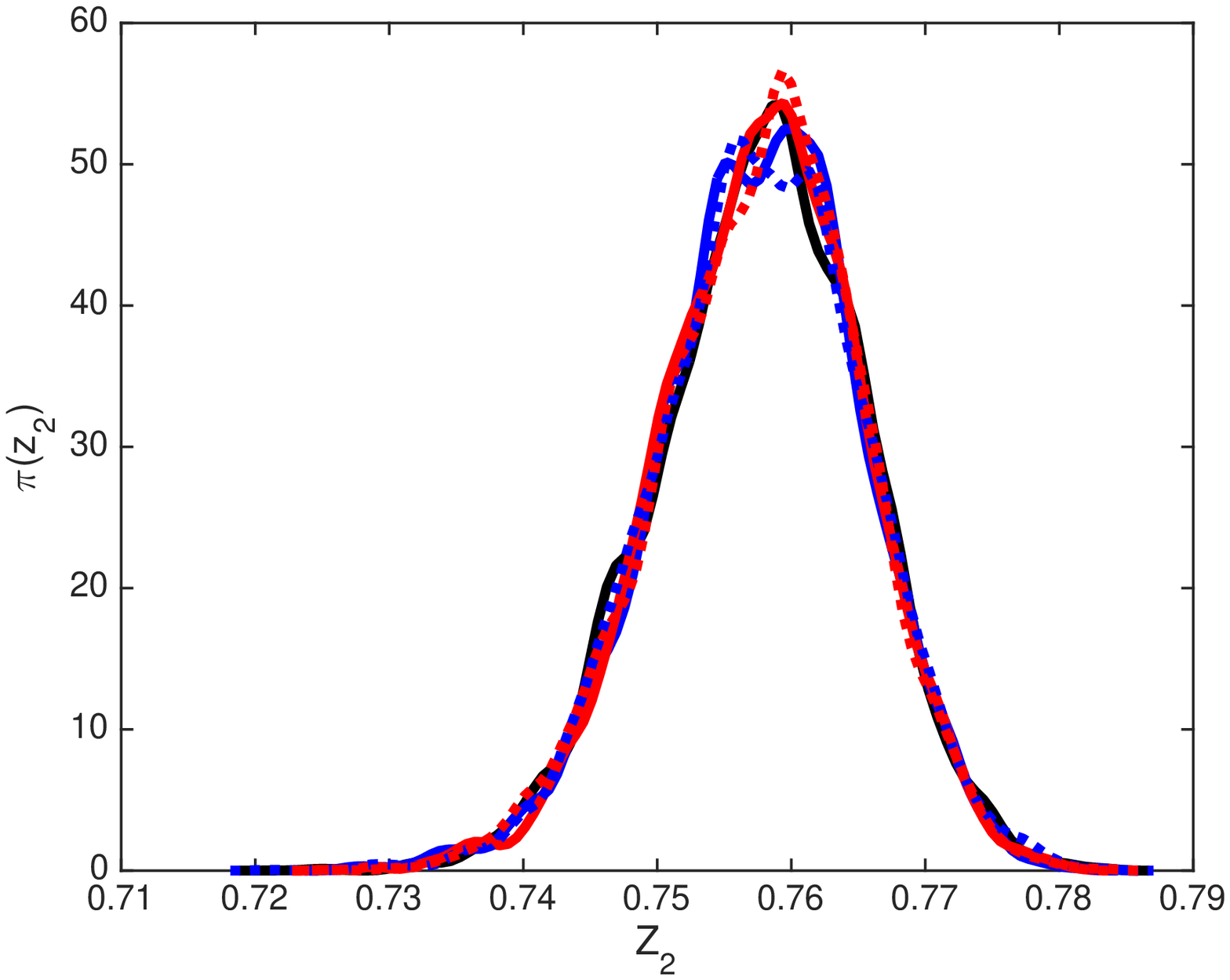}
  \end{overpic}
    \begin{overpic}[width=0.45\textwidth,trim= 20 0 20 15, clip=true,tics=10]{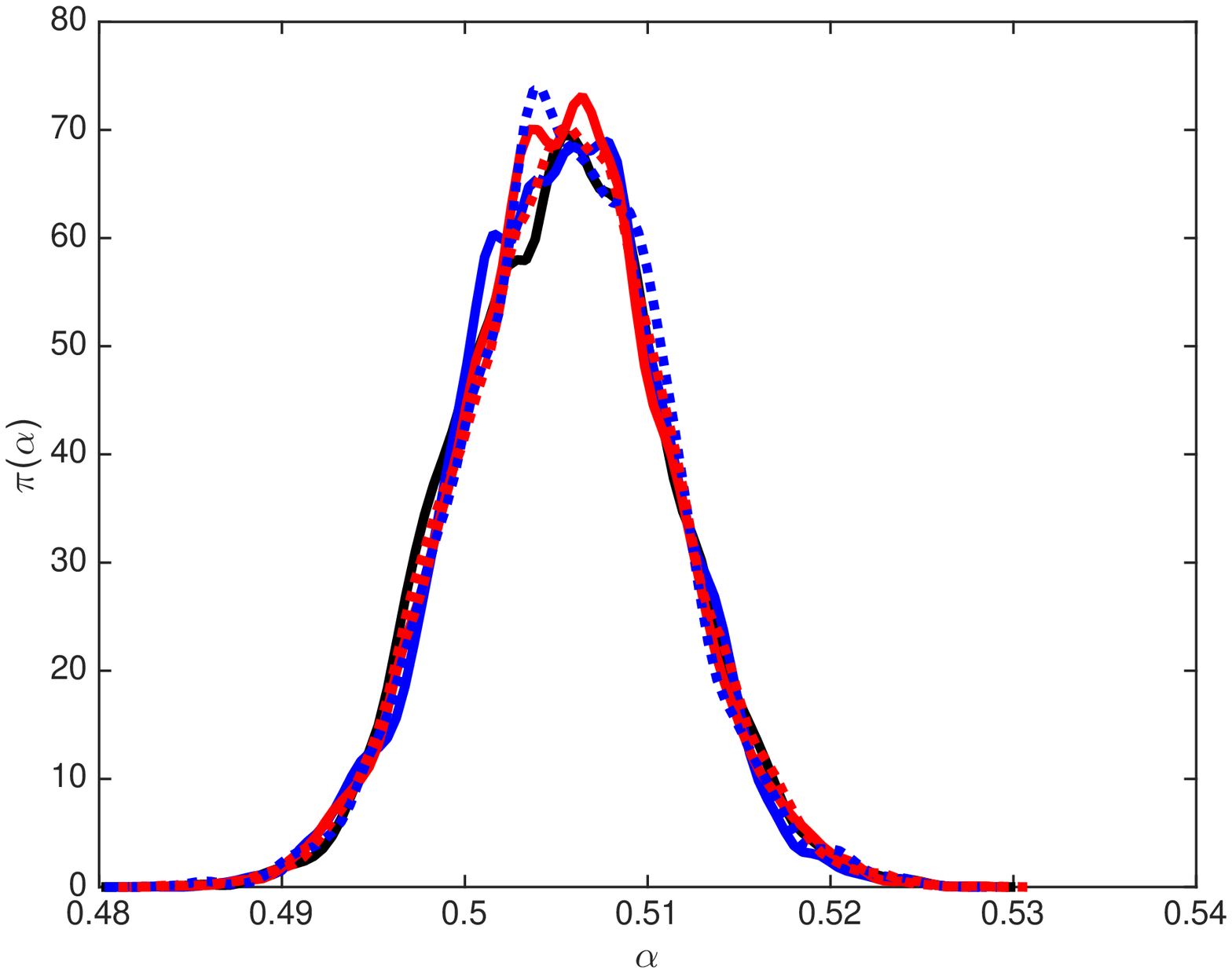}
  \end{overpic}
    \begin{overpic}[width=0.45\textwidth,trim= 20 0 20 15, clip=true,tics=10]{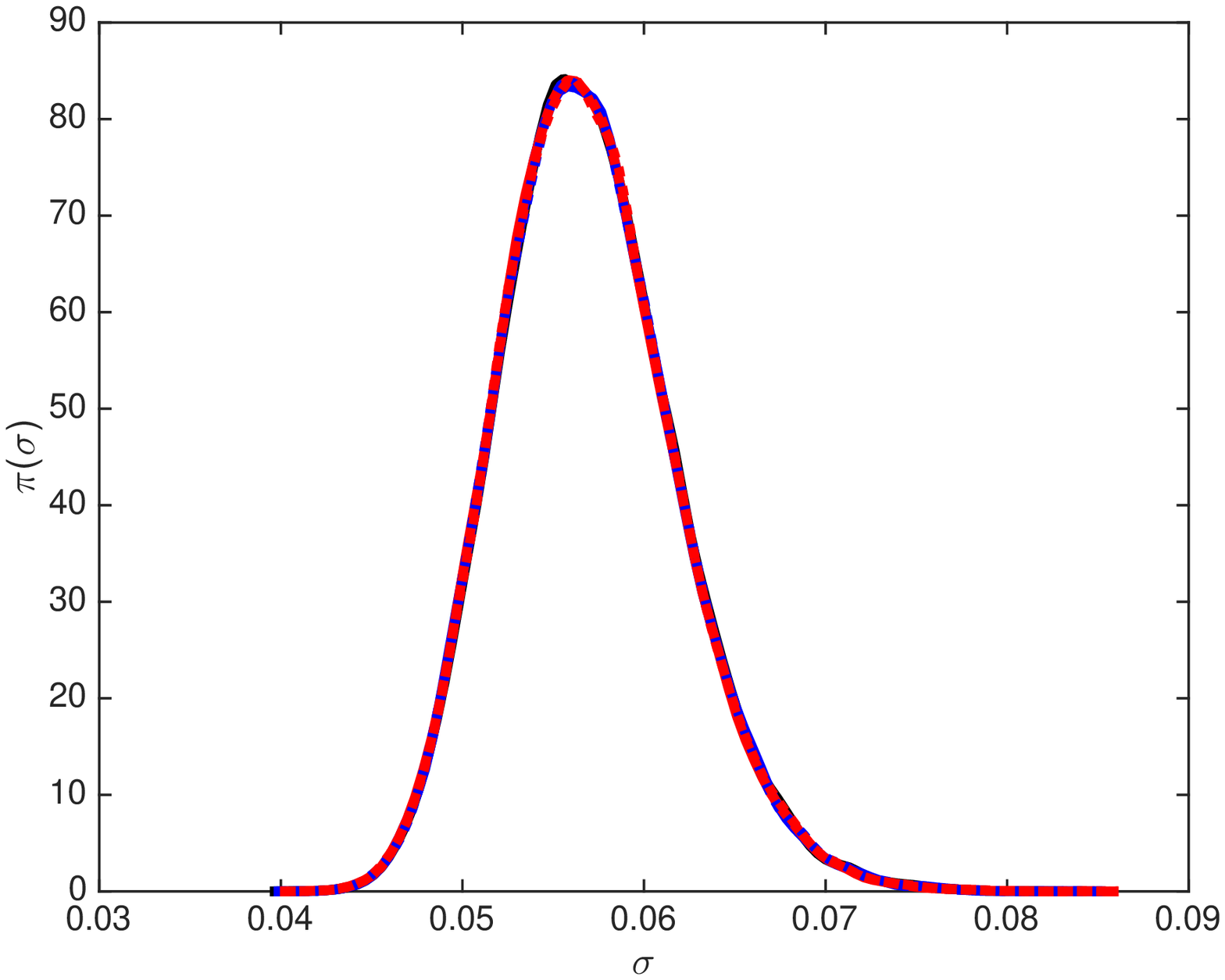}
  \end{overpic}
\end{center}
\caption{Example 1 ($\alpha$ is unknown): The marginal distribution of each component of the parameter. Densities result from AMPC approach; compared to direct method.}\label{pi-eg2-tol}
  \end{figure}

  \begin{figure}
\begin{center}
    \begin{overpic}[width=0.45\textwidth,trim= 20 0 20 15, clip=true,tics=10]{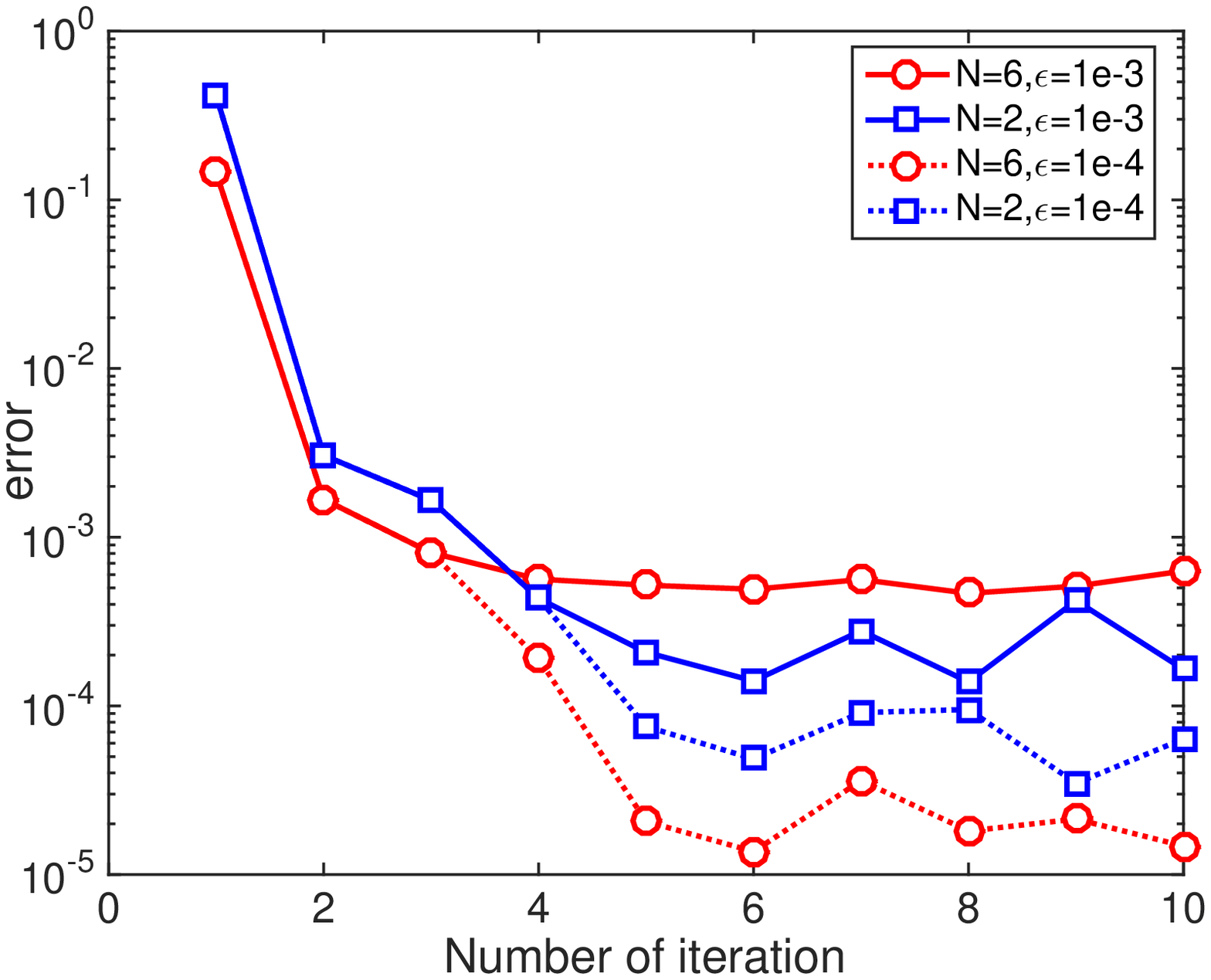}
  \end{overpic}
\end{center}
\caption{Example 1 ($\alpha$ is unknown): the relative errors with number of iteration for AMPC approach.}\label{error-f2}
\end{figure}

In this section, we assume that the fractional order $\alpha$ in the model (\ref{2dsource}) is unknown. Thus the purpose is to recover is $Z=(Z_0,Z_1,\alpha)\in \mathbb{R}^3$. We set a uniform prior for $Z=(Z_0,Z_1,\alpha)$ on the parameter space $[0,1]^3$.  To generate data for the inversion, the `true' parameter is choose to be $(Z_0=0.25, Z_1=0.75, \alpha=0.5)$. In our computations, the measurements are taken at three different times $t=\{0.25, 0.75,1\}$ using 25 locations on a uniform $5\times5$ grid covering the domain $D$, leading to $n_d=75$ measurements.  The initial guess for $Z$ is set as (0, 0, 0.1).

The posterior densities of components $Z=\{Z_0,Z_1,\alpha\}$ with $\sigma =0.05$ noise  are shown in Figure \ref{pi-eg2}.  As expected, a poor estimate is obtained by the prior-based PC approach with a lower order $N$, while the AMPC approach admits a much improved result. The computational times given by three different approaches are presented in Table \ref{eg1_time_2}. For the prior-based PC approach, we notice that even with $N=6$ (meaning that the number of model evaluations is about two times than that of the AMPC approach with $N=2$), it failed to obtain an accurate approximation as in the AMPC approach (see Figure \ref{pi-eg2}). Meanwhile, the results by AMPC approach agree well with the conventional MCMC approach, but its computational time is only a very small portion of that of the conventional MCMC approach. To see this, with $N=2$ and \textcolor{black}{$Q=80$} online high-fidelity evaluation,  the AMPC approach uses only \textcolor{black}{$13.5s$} for $5\times 10^4$ MCMC iterations. In contrast, for the same number of MCMC iterations the conventional MCMC approach needs $4483.7s \approx 1.25\mathrm{h}$.  Therefore, the speedup of the AMPC approach over the conventional approach is dramatic. Also, the AMPC approach with $N=6$ (\textcolor{black}{$21.2s$}) and $N=2$ (\textcolor{black}{$13.5s$}) is significantly faster  for online computations compared to the prior-based PC approach with $N=10$ ($26.7s$).

We now investigate the sensitivity of the numerical results with respect to the threshold $\epsilon$. The one dimension marginal posterior distributions by the AMPC approach are illustrated in Figure \ref{pi-eg2-tol}. The relative errors $err$ with the number of iteration for AMPC approach are also shown in Figure \ref{error-f2}. These results show that the AMPC approach admits reasonably accurate results   for different values of $\epsilon$.

 \subsection{Example 2: An elliptic PDE inverse problem}
Let $\Omega=[0,1]^2$, and we consider the following elliptic PDE
\begin{eqnarray}\label{2dellip}
\begin{array}{rl}
-\nabla\cdot(\kappa(x) \nabla u(x))&=f(x),\quad x\in D,\\
 u(x)&=0, \quad  \quad \,\, x\in \partial{D}.
 \end{array}
\end{eqnarray}
We set the source as $f(x) = 100\sin(\pi x_1)\sin(\pi x_2)$. The goal here is to recover the permeability $\kappa(x)$ from a set of noise measurements of $u.$

In our simulations, the permeability field is projected onto a set of radial basis functions (RBF), and hence the inference is carried out on the weights $\{\kappa_i\}$ of the RBF bases:
\begin{eqnarray*}
\kappa(x)=\sum^{9}_{i=1}\kappa_i \exp\left(- 0.5\frac{\|x-x_{0,i}\|^2}{0.15^2}\right),
\end{eqnarray*}
where $\{x_{0,i}\}^{9}_{i=1}$ are the centers of the RBF.  The prior distributions on each of the weights $\kappa_i, i=1,\cdots, 9$ are assumed to be independent and log-normal, i.e., $\log(\kappa_i)\sim N(0,1)$.  To avoid the inverse crime, the true weight is drawn from  $\log(\kappa_i)\sim U(-5,5)$.  The measurements of $u$ are generated by selecting  the values of the states at 81 measurement sensors evenly distributed over $[0.1,0.9]^2$ with grid spacing 0.1.  The true permeability field that is used to generate the test data and the corresponding output data are shown in Fig.\ref{set_eg3}. The observational errors are set to be additive and Gaussian:
\begin{eqnarray*}
d_j=u(x_j)+\max_{j}\{|u(x_j)|\}\delta\xi_j,
\end{eqnarray*}
where $\delta$ dictates the noise level and $\xi_j$  is a  Gaussian random variable with zero mean and unit standard deviation. Notice that in this example, the parameters are far away from the prior, and thus one would expect a bad approximation with the prior-based PC approach due to the lack of global accuracy of the PC surrogate.

\begin{table}[tp]
      \caption{Example 2. Computational times, in seconds, given by three different methods. $\epsilon=1\times10^{-3}, \delta=0.05$.}\label{eg2_time}
  \centering
  \fontsize{6}{12}\selectfont
  \begin{threeparttable}
    \begin{tabular}{ c ccccc}
  \toprule

   & \multicolumn{2}{c}{Offline}&\multicolumn{2}{c}{Online}\cr
\cmidrule(lr){2-3} \cmidrule(lr){4-5}

  \multirow{1}{*}{Method}  &$\text{$\#$ of model evaluations}$&CPU(s) &$\text{$\#$ of model evaluations}$&CPU(s)     &\multirow{1}{*}{Total time(s)}\cr
  \midrule
    Direct                                     & $-$       & $-$         & 5$\times 10^4$    &$\text{\bf{1492.2}}$      & $\text{\bf{1492.2}}$   \cr
   PC, $N=7$                & $\text{\bf{28,880}}$  & $\text{\bf{1246.7}}$    & $-$                       &$\text{\bf{241.2}}$      & $\text{\bf{1487.9}}$   \cr
   PC, $N=6$                & 10,010  & 340.4     & $-$                       &121.7                         & 426.1  \cr
   PC, $N=4$                & 1,430   & 37.5       & $-$                        & 34.3                         & 71.8  \cr
   PC, $N=2$                & 110      & 2.9         &$-$                         & 11.4                          & 14.3  \cr
   AMPC, $N=4, N_C=2$   & 1,430    & 37.5       & 570                     & 49.2     & 86.7   \cr
  AMPC, $N=2, N_C=2$        & $\text{\bf{110}}$ & $\text{\bf{2.9}}$ & $\text{\bf{1,010}}$                   & $\text{\bf{38.7}}$       & $\text{\bf{41.6}}$   \cr
    \bottomrule
      \end{tabular}
    \end{threeparttable}

\end{table}

  \begin{figure}
\begin{center}
    \begin{overpic}[width=0.45\textwidth,trim= 20 0 20 15, clip=true,tics=10]{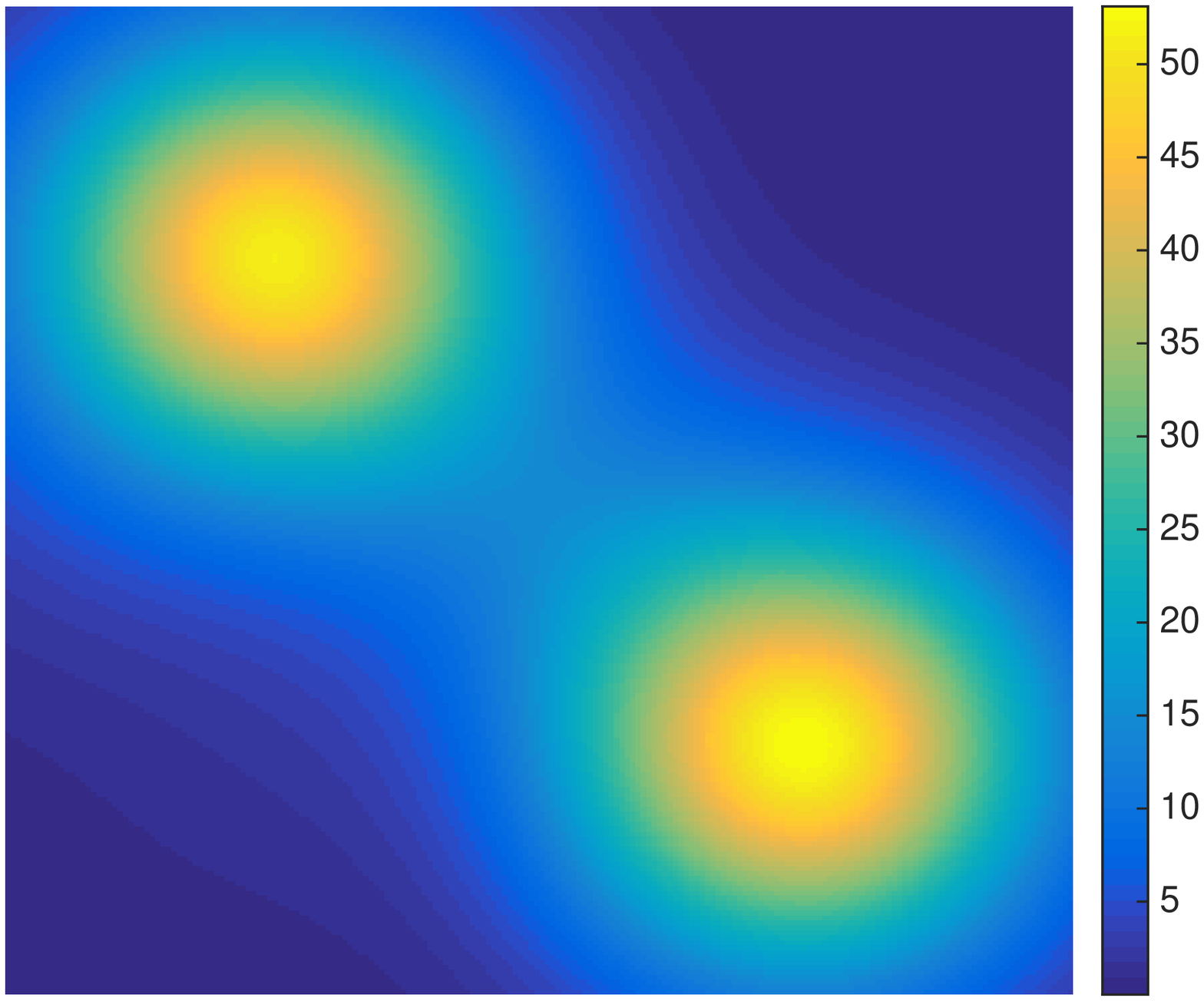}
      \end{overpic}
   \begin{overpic}[width=0.45\textwidth,trim= 20 0 20 15, clip=true,tics=10]{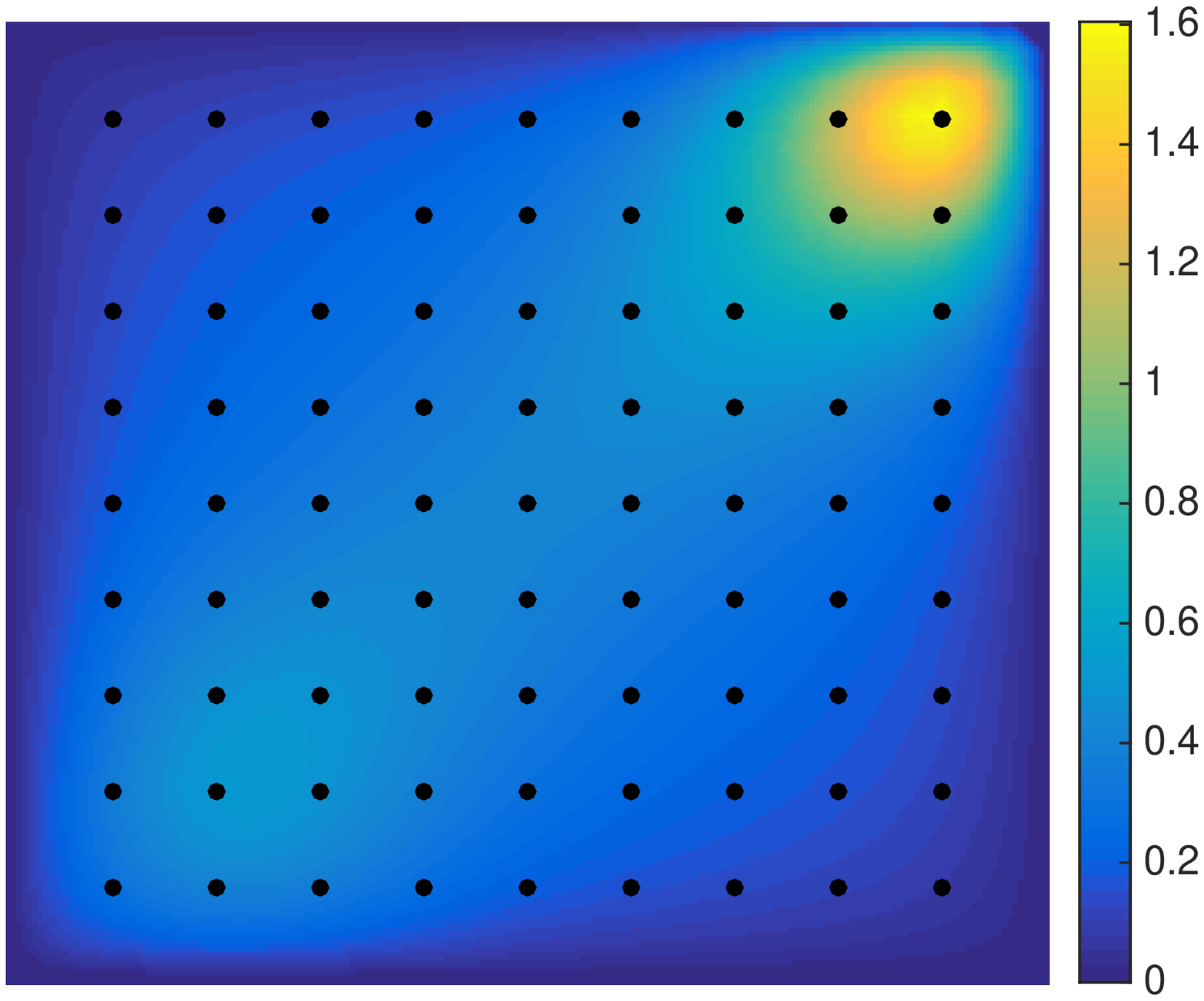}
  \end{overpic}
\end{center}
\caption{Example 2:  Set up. Left: the true permeability used for generating the synthetic data sets. Right: the model outputs of the permeability.  }\label{set_eg3}
\end{figure}

  \begin{figure}
\begin{center}
  \begin{overpic}[width=0.32\textwidth,trim=20 0 20 15, clip=true,tics=10]{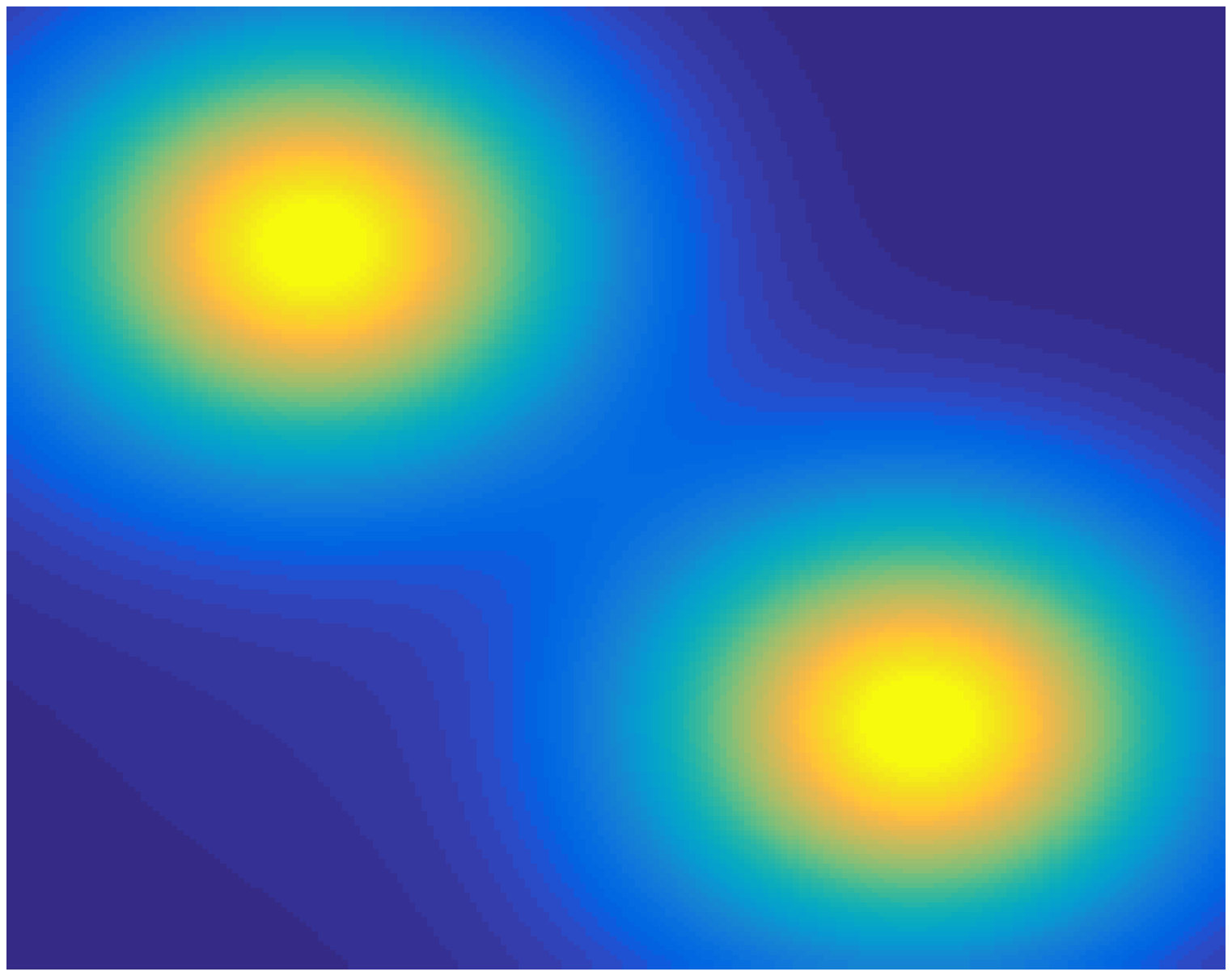}
  \end{overpic}
    \begin{overpic}[width=0.32\textwidth,trim= 20 0 20 15, clip=true,tics=10]{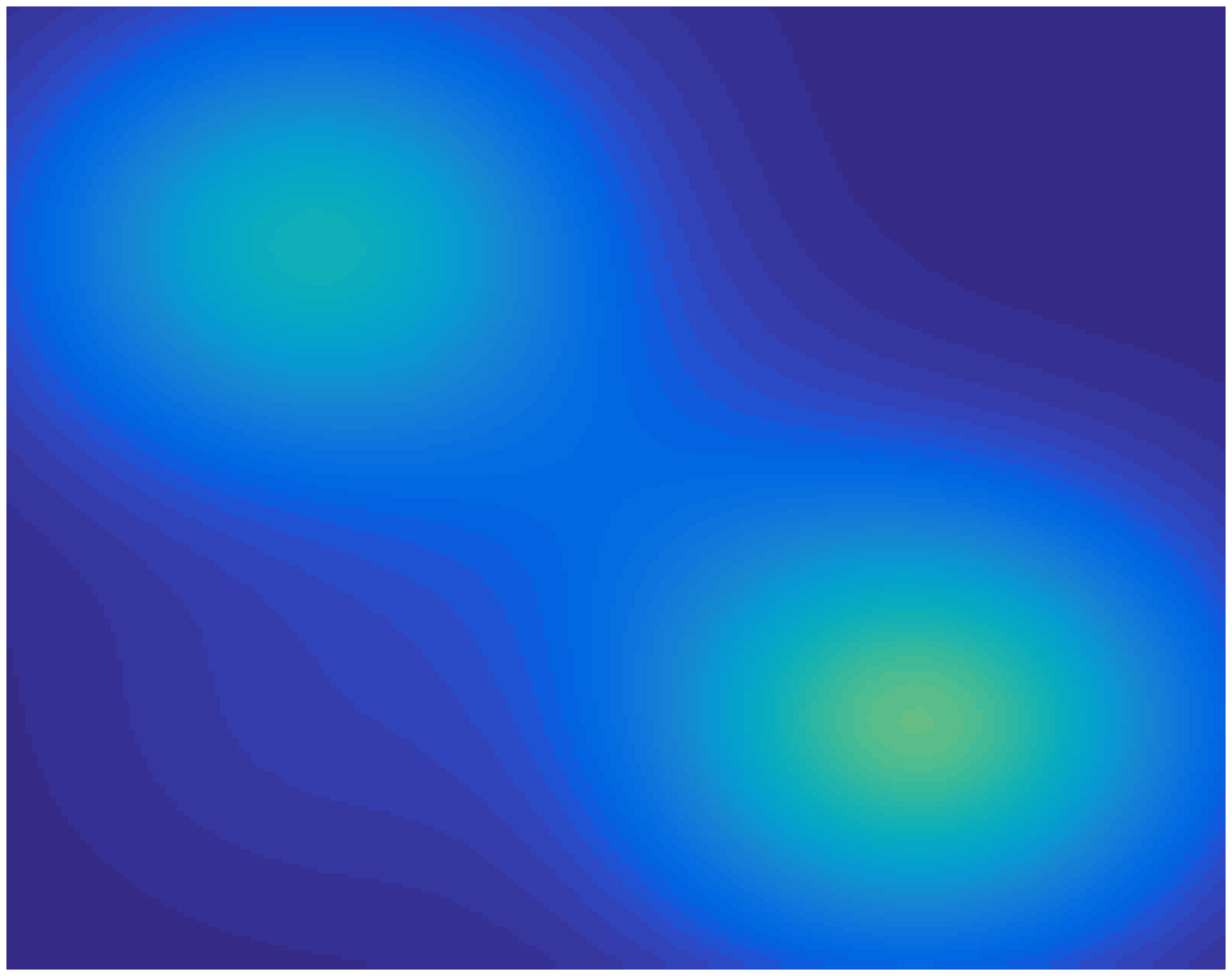}
  \end{overpic}
    \begin{overpic}[width=0.32\textwidth,trim= 20 0 20 15, clip=true,tics=10]{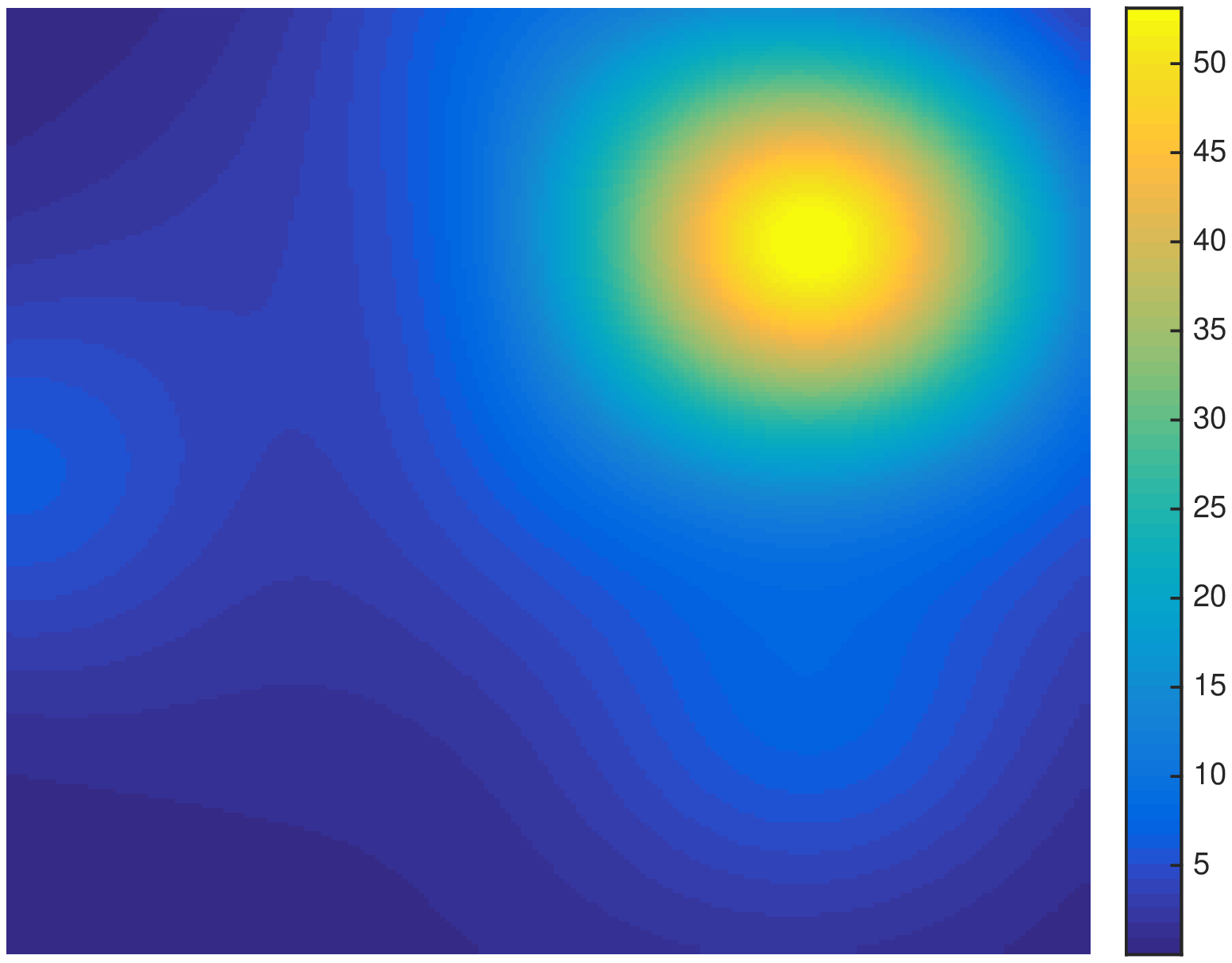}
  \end{overpic}
  \begin{overpic}[width=0.32\textwidth,trim=20 0 20 15, clip=true,tics=10]{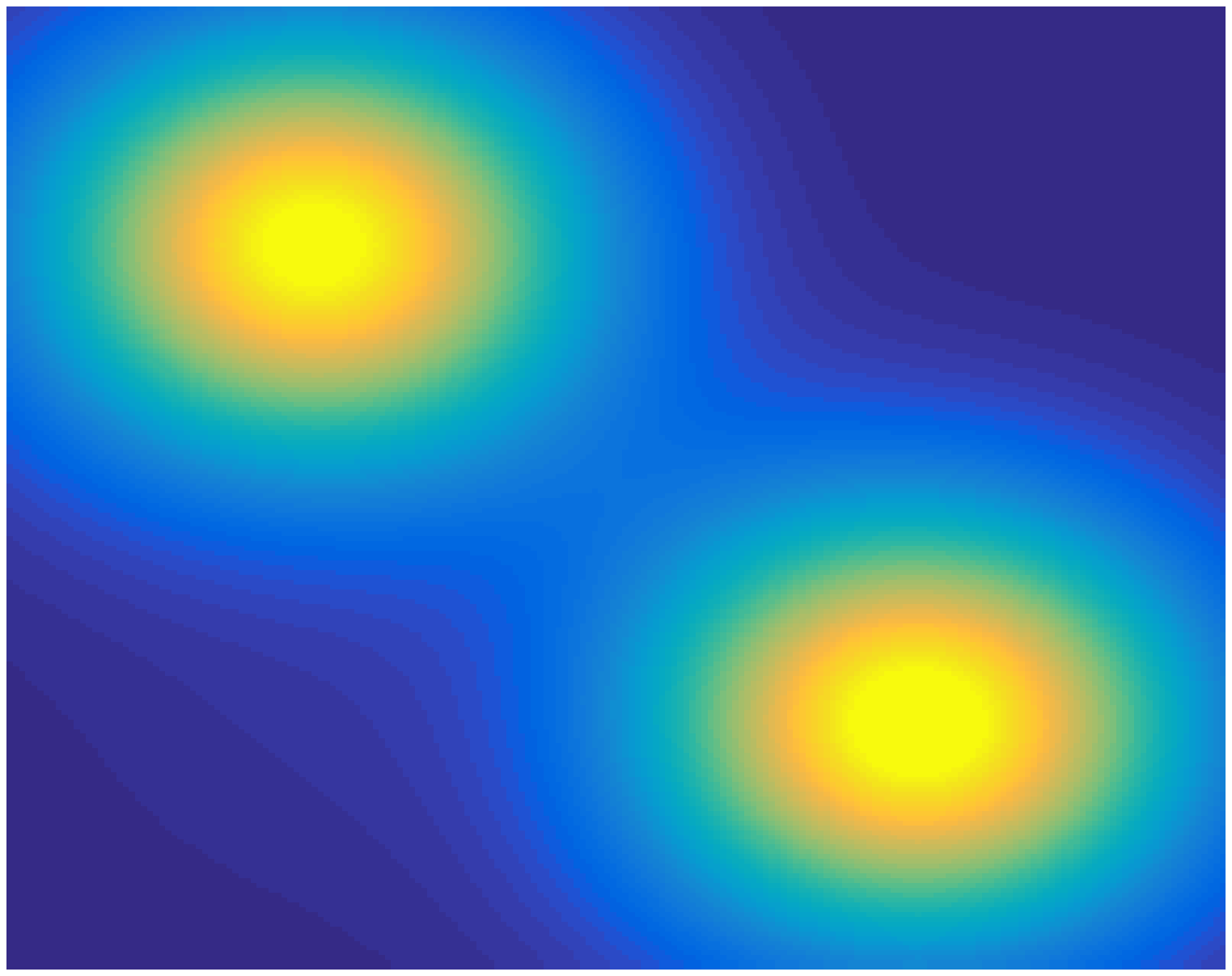}
  \end{overpic}
    \begin{overpic}[width=0.32\textwidth,trim= 20 0 20 15, clip=true,tics=10]{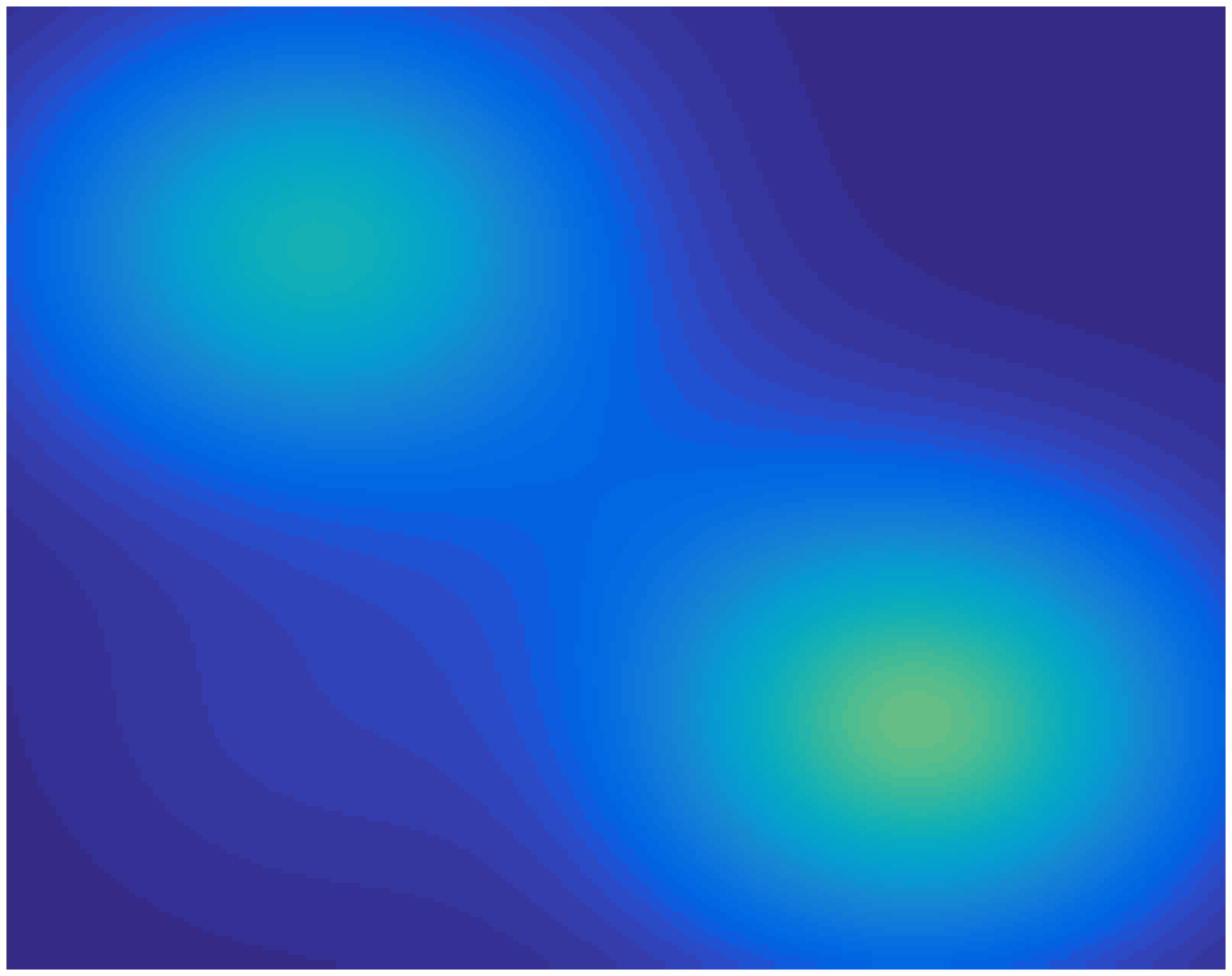}
  \end{overpic}
    \begin{overpic}[width=0.32\textwidth,trim= 20 0 20 15, clip=true,tics=10]{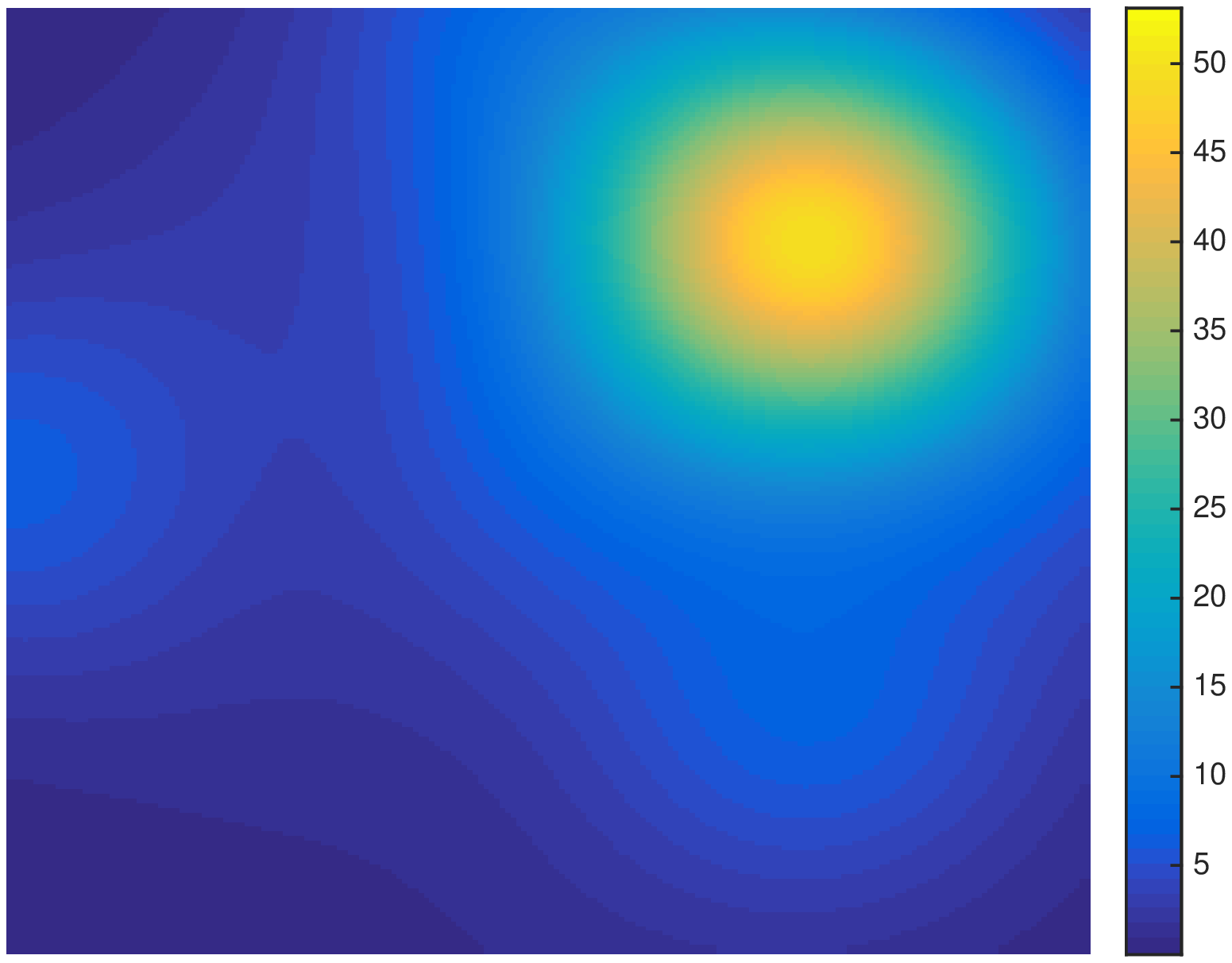}
  \end{overpic}
\end{center}
\caption{Example 2:  (Left column) Conditional mean arising from full model. (Middle column)  Conditional mean arising from prior-based PC model ($N=7$). (Right column) Conditional mean arising from  prior-based PC model ($N=2$). From top to bottom, the relative noise  level $\delta$ is $0.01,0.05$ respectively.}\label{prior_sol_eg2}
  \end{figure}

  \begin{figure}
\begin{center}
  \begin{overpic}[width=0.32\textwidth,trim=20 0 20 15, clip=true,tics=10]{u_D_s01_eg2.eps}
  \end{overpic}
    \begin{overpic}[width=0.32\textwidth,trim= 20 0 20 15, clip=true,tics=10]{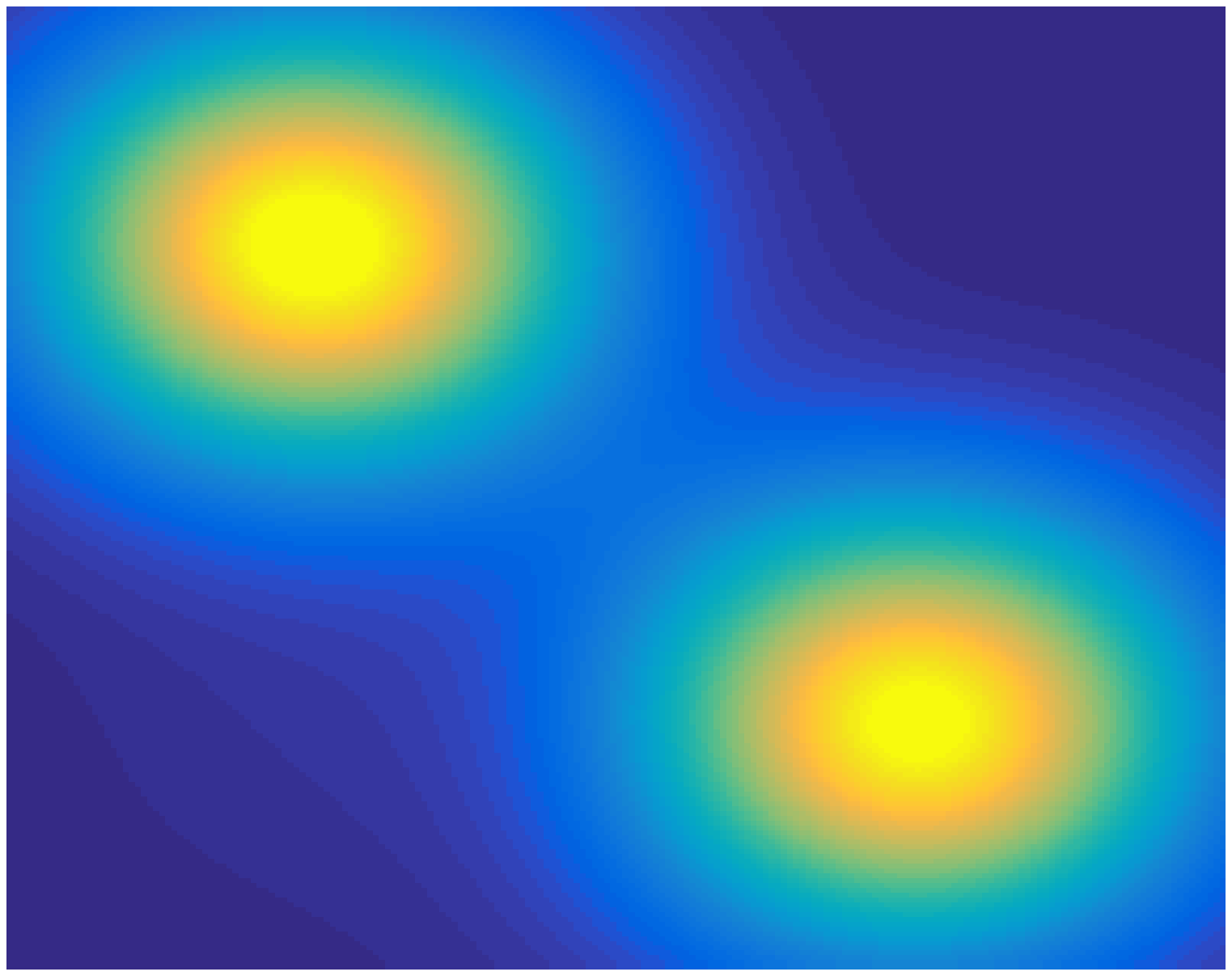}
  \end{overpic}
    \begin{overpic}[width=0.32\textwidth,trim= 20 0 20 15, clip=true,tics=10]{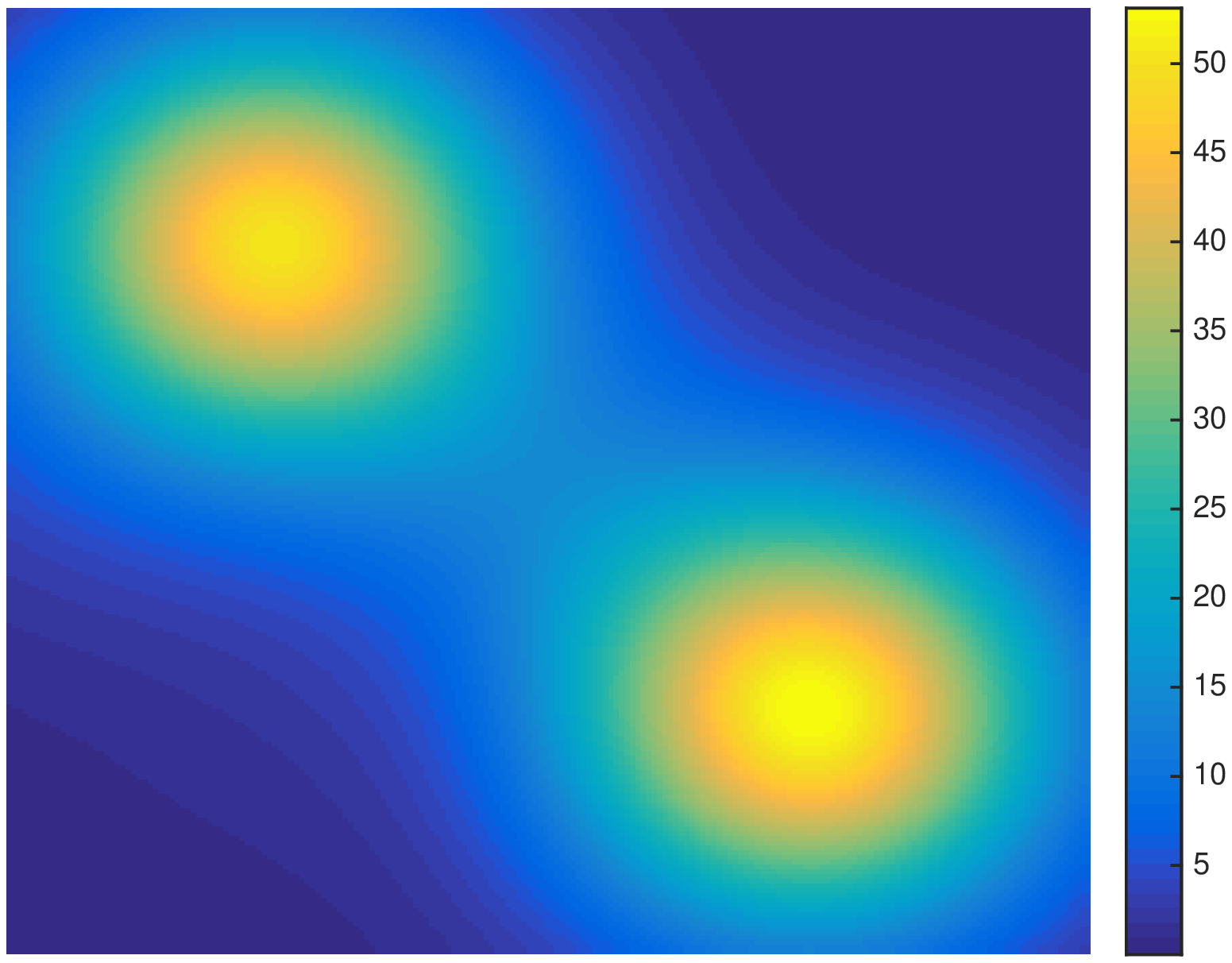}
  \end{overpic}
  \begin{overpic}[width=0.32\textwidth,trim=20 0 20 15, clip=true,tics=10]{u_D_s05_eg2.eps}
  \end{overpic}
    \begin{overpic}[width=0.32\textwidth,trim= 20 0 20 15, clip=true,tics=10]{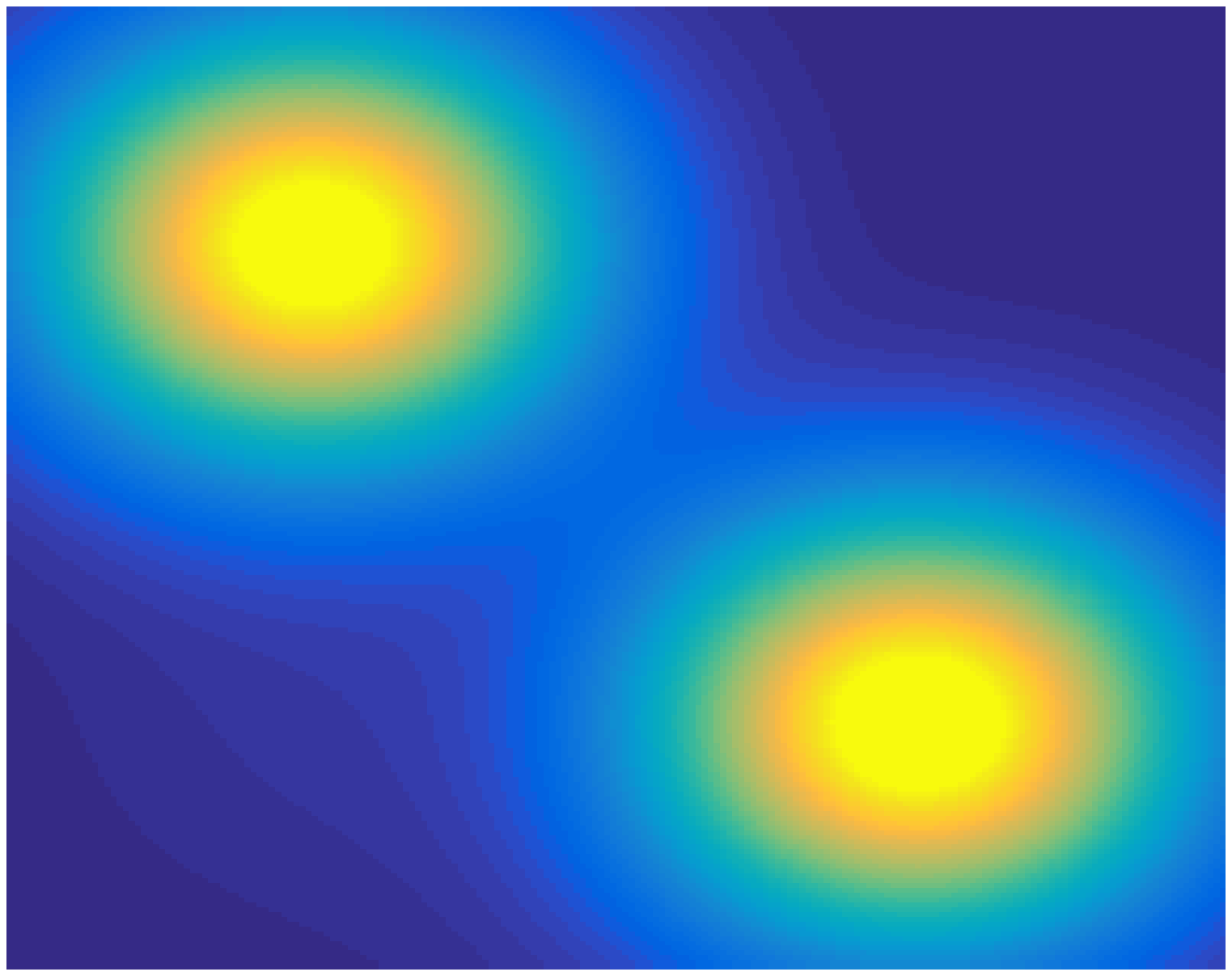}
  \end{overpic}
    \begin{overpic}[width=0.32\textwidth,trim= 20 0 20 15, clip=true,tics=10]{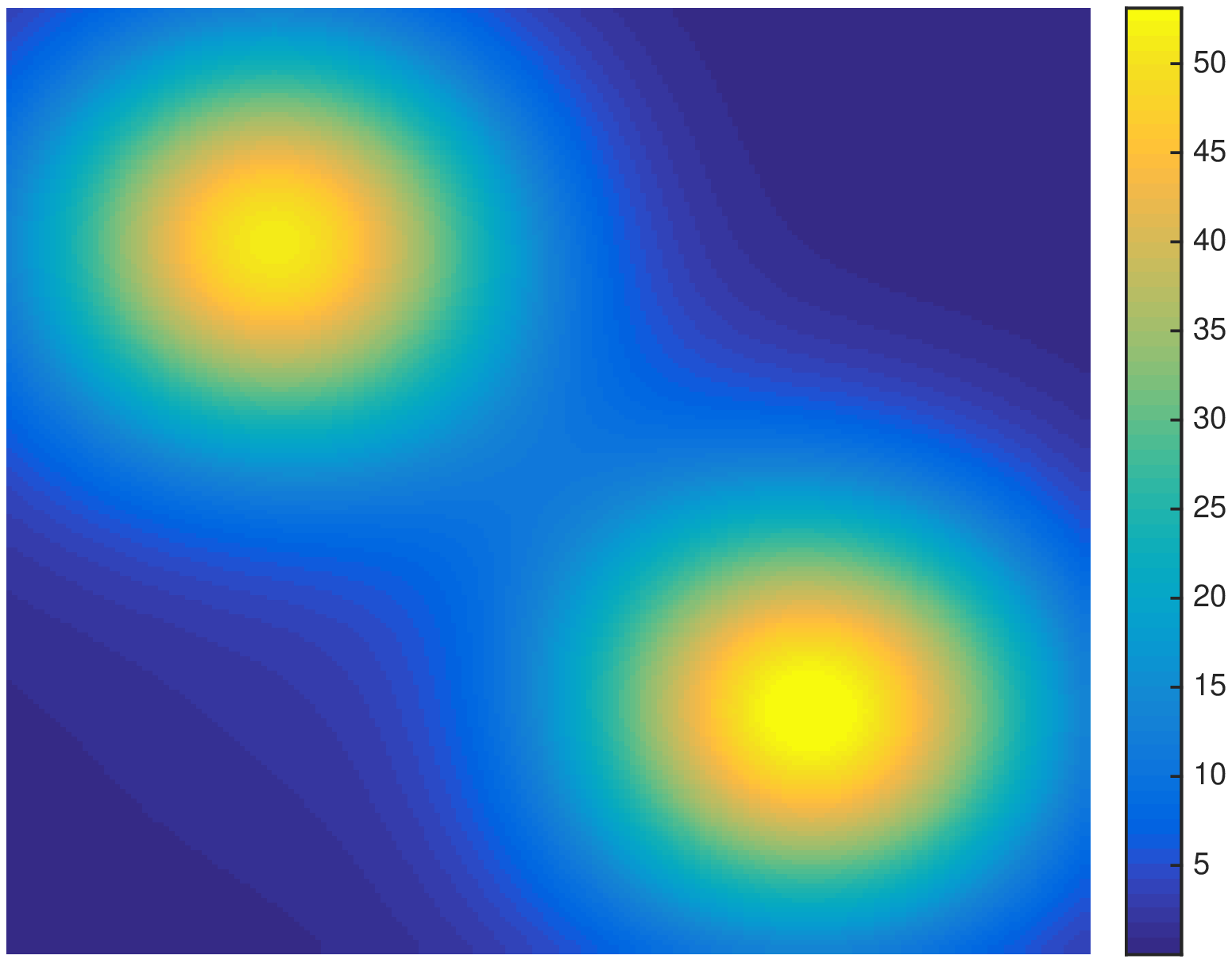}
  \end{overpic}
\end{center}
\caption{Example 2:  (Left column) Conditional mean arising from full model. (Middle column)  Conditional mean arising from AMPC model ($N=4$). (Right column) Conditional mean arising from AMPC model ($N=2$). From top to bottom, the relative noise  level $\delta$ is $0.01,0.05$ respectively.}\label{adap_sol_eg2}
  \end{figure}

The comparison results between the conventional MCMC and the  prior-based PC approach is shown in Figure \ref{prior_sol_eg2}. It is clearly seen that the prior-based PC approach admits very large approximation error, due to the fact that the parameter is far away from what is assumed in the prior. The CPU time of evaluating the conventional MCMC is 1492.2s (see Table 4.3), while the CPU time of prior-based PC method with $N=2$ is about 11.4s. Although the a prior-based PC approach can gain the computational efficiency, the estimation accuracy cannot be guaranteed. To improve this, one can increase the PC order $N$. However, when the order increases, the cost of constructing the prior-based PC surrogate becomes increasingly expensive for this high dimensional problem. To see this, consider the case with $N=7,$ one requires $28,880$ offline model evaluations -- resulting in an offline CPU time 1246.7s.

Similar comparison results between the conventional MCMC and the AMPC approach are shown in Figure \ref{adap_sol_eg2}. It is not surprising that even with a lower PC order $N=2,$ the AMPC approach admits a rather accurate result. As shown in Figure \ref{adap_sol_eg2}, the conditional means obtained by the conventional MCMC and the AMPC approach agree very well. This confirms the accuracy of the AMPC approach.

In Table \ref{eg2_time}, we also learn that building a prior-based PC surrogate with $N=2$ (resp. $N=7$) requires an offline CPU time of 2.9s (resp. 1246.7s), whereas its online evaluation requires 11.4s (resp. 241.2s). This reflects the major drawback for the prior-based PC approach for high-dimensional problems: the overall CPU time increase fast with respect to the polynomial order $N.$  In contrast, for the AMPC approach with $N=2$, the offline and online CPU times are 2.9s and 38.7s, respectively, meaning that the  AMPC approach can provide with much more accurate results, yet with less computational time. This indicates that the AMPC approach is more efficient than the prior-based PC approach for solving high dimension problems.

\section{Summary} \label{sec:summary}

\textcolor{black}{We have developed an efficient adaptive multi-fidelity approach for solving  Bayesian  inverse problems when the forward model evaluation is computationally expensive. The algorithm introduces a multi-fidelity PC model to approximate the computationally expensive forward model and refines the approximation incrementally.  Since the high-fidelity model evaluations are only needed to construct the lower order multi-fidelity PC, the computational cost can be significantly reduced.  Further, the algorithm refines the multi-fidelity model over a sequence of samples adaptively determined from data so that the approximation can eventually concentrate on the posterior distribution. In typical inference problems, when the data are informative, the adaptive approach can lead to significant gains in efficiency and accuracy over previous prior-based PC methods.  A rigorous error analysis of the algorithm has been discussed. In particular,  we analyze the error bounds on the Kullback-Leibler distance between the true posterior distribution and the approximation based on surrogate model. }

\textcolor{black}{
The performance of the proposed strategy are demonstrated on two nonlinear inverse problems: estimating the source location for time-fractional diffusion equations and inferring the permeability field for elliptic PDEs. Numerical experiments confirm that the proposed AMPC methods are efficient (as measured by the number of high-fidelity model evaluations)  and accurate (as measured by posterior mean and information divergence from the exact posterior) in comparison with prior-based PC approach. When compared with the conventional MCMC, AMPC is still favorable because of it can reduce the cost of each posterior evaluation by server orders of magnitude. Since there is no conflict between the AMPC methods and other varieties of Monte Carlo scheme; such as Sequential Monte Carlo (SCM) or hybrid Monte Carlo (HMC) algorithm, the two ideas can work together and further improve the robustness of AMPC methods. We leave this to our next exploration.}

\section*{Acknowledgment}
The authors would like to thank Tiangang Cui for many comments and discussions.

\appendix
\textcolor{black}{
\section{ Proof of Theorem \ref{t1}} 
This section is devoted to proving our main theoretical result. }

\textcolor{black}{
In order to proof the Theorem \ref{t1}, we first define the potential $\Psi$ given by
\begin{equation}\label{potential}
\Psi(z;d)=\frac{\|u^H(z)-d\|^2}{2\sigma^2},
\end{equation}
where $\sigma$ is the known standard deviation of the noise $e$.
Using the approximation $u^L$ of the high-fidelity model $u^H$, we can define the approximation $\widetilde{\Psi}_N(z;d)$ given by $\widetilde{\Psi}_N(z;d)=\frac{\|u^L(z)-d\|^2}{2\sigma^2}$. We also define the normalizing constants $\gamma$ and $\widetilde{\gamma}_N$ as  $$\gamma=\int_{\Gamma} \mathcal{L}(z)\pi(z)dz,\,\,\,\,\widetilde{\gamma}_N=\int_{\Gamma} \widetilde{\mathcal{L}}_N(z)\pi(z)dz,$$
where $$\mathcal{L}(z)=\exp(-\Psi(z;d)), \widetilde{\mathcal{L}}_N(z)=\exp(-\widetilde{\Psi}_N(z;d)).$$ 
}

\textcolor{black}{
We have the following lemma concerning the approximation $\widetilde{\Psi}_N(z;d)$.
\begin{lemma}[\cite{Cui2014data}]\label{lemmaA1}
Assume the functions $u^H$ and $u^L$ satisfy Assumption \ref{a1} uniformly in $N$. For a given $\epsilon>0$, \\
 \noindent (i) there exist a constant $K>0$ such that $|\Psi-\widetilde{\Psi}_N|\leq K \epsilon, \forall z\in \Gamma_N(\epsilon).$\\
 \noindent (ii) there exist  constants $k_1>0$ and $k_2>0$ such that $$|\gamma-\widetilde{\gamma}_N|\leq k_1\epsilon+k_2 \mu\big(\Gamma^{\perp}_N(\epsilon)\big).$$
\end{lemma}
}

\textcolor{black}{
We are now ready to prove bounds on the approximation error in the posterior distributions, i.e. Theorem \ref{t1}.
}
\begin{proof} 
\textcolor{black}{
Since $\sup_{z\in \Gamma} |\Psi(z;d)|$ is bounded when $\sup_{z\in \Gamma} \|u^H(z)\|$ is bounded, we have
\begin{equation*}
1\geq \gamma= \int_{\Gamma}\exp(- \Psi(z;d))\pi(z)dz \geq\int_{\Gamma}\exp(-\sup_{z\in \Gamma} |\Psi(z;d)|)\pi(z)dz\geq \exp(-C_{\Psi}),
\end{equation*}
where $C_{\Psi}=\sup_{z\in \Gamma} |\Psi(z;d)|.$
We have an analogous  bound for $\widetilde{\gamma}_N$: $1\geq\widetilde{\gamma}_N \geq \exp(-C_{\widetilde{\Psi}_N})$.
}

\textcolor{black}{
Similar to the derivation in \cite{Yan+Zhang2017IP}, we have
\begin{eqnarray*}
D_{KL}(\widetilde{\pi}_N^d||\pi^d)&\leq& D_{KL}(\widetilde{\pi}_N^d||\pi^d)+D_{KL}(\pi^d||\widetilde{\pi}_N^d)\\
&\leq&\frac{\gamma-\widetilde{\gamma}_N}{\gamma\widetilde{\gamma}_N}\int_{\Gamma} \mathcal{L}(\Psi-\widetilde{\Psi}_N)\pi(z)dz+\frac{1}{\widetilde{\gamma}_N}\int_{\Gamma}|\widetilde{\mathcal{L}}_N-\mathcal{L}||\Psi-\widetilde{\Psi}_N|\pi(z)dz\\
&:=&I_1+I_2,
\end{eqnarray*}
where
$$I_1=\frac{\gamma-\widetilde{\gamma}_N}{\gamma\widetilde{\gamma}_N}\int_{\Gamma} \mathcal{L}(\Psi-\widetilde{\Psi}_N)\pi(z)dz$$
and
$$I_2=\frac{1}{\widetilde{\gamma}_N}\int_{\Gamma}|\widetilde{\mathcal{L}}_N-\mathcal{L}||\Psi-\widetilde{\Psi}_N|\pi(z)dz.$$
}

\textcolor{black}{
For the first term, since $\Psi(z;d)$  defined in (\ref{potential}) is non-negative, we have
\begin{eqnarray*}
&&\int_{\Gamma}  \mathcal{L}|\Psi-\widetilde{\Psi}_N|\pi(z)dz=\int_{\Gamma} \exp(-\Psi)|\Psi-\widetilde{\Psi}_N|\pi(z)dz\\
&&\leq \int_{\Gamma_N(\epsilon)} |\Psi-\widetilde{\Psi}_N|\pi(z)dz+\int_{\Gamma^{\perp}_N(\epsilon)}  \gamma |\Psi-\widetilde{\Psi}_N|\pi^d(z)dz\\
&&\leq c_1 \epsilon+ \sup_{\Gamma}(|\Psi-\widetilde{\Psi}_N|)\int_{\Gamma^{\perp}_N(\epsilon)}  \pi^d(z)dz\\
&&\leq c_1 \epsilon+c_2\mu\big(\Gamma^{\perp}_N(\epsilon)\big),
\end{eqnarray*}
for constants $c_1$ and $c_2$, independent of  $N$.
}

\textcolor{black}{
Combine the above results and Lemma \ref{lemmaA1}, we have
\begin{eqnarray*}
I_1&=\frac{\gamma-\widetilde{\gamma}_N}{\gamma\widetilde{\gamma}_N}\int_{\Gamma} \mathcal{L}(\Psi-\widetilde{\Psi}_N)\pi(z)dz\\
 &\leq \frac{1}{\exp(-C_\Psi-C_{\widetilde{\Psi}_N})} \Big(K_3 \epsilon+K_4\mu\big(\Gamma^{\perp}_N(\epsilon)\big)\Big)^2.
\end{eqnarray*}
}

\textcolor{black}{
Furthermore, we use the lower bound of $\widetilde{\gamma}_N$, together with the local Lipschitz continuity of the exponential function to bound
\begin{eqnarray*}
I_2&=&\frac{1}{\widetilde{\gamma}_N}\int_{\Gamma}|\widetilde{\mathcal{L}}_N-\mathcal{L}||\Psi-\widetilde{\Psi}_N|\pi(z)dz\\
&\leq&\frac{1}{\exp(-C_{\widetilde{\Psi}_N})}\int_{\Gamma_N(\epsilon)}|\Psi-\widetilde{\Psi}_N|^2\pi(z)dz+\frac{\gamma}{\widetilde{\gamma}_N}\int_{\Gamma^{\perp}_N(\epsilon)}|1-\exp(\Psi-\widetilde{\Psi}_N)||\Psi-\widetilde{\Psi}_N|\pi^d(z)dz\\
&\leq&c_3\epsilon^2+c_4\mu\big(\Gamma^{\perp}_N(\epsilon)\big),
\end{eqnarray*}
for constants $c_3$ and $c_4$, independent of  $N$.
}

\textcolor{black}{
Note that the bounds on $I_1$ and $I_2$ from above are independent of $N$, we have
\begin{eqnarray*}
D_{KL}(\widetilde{\pi}_N^d||\pi^d)&\leq&C\Big(K_3 \epsilon+K_4\mu\big(\Gamma^{\perp}_N(\epsilon)\big)\Big)^2+c_3\epsilon^2+c_4\mu\big(\Gamma^{\perp}_N(\epsilon)\big) \\
&\leq& \Big(K_1 \epsilon+K_1\mu\big(\Gamma^{\perp}_N(\epsilon)\big)\Big)^2
\end{eqnarray*}
with  constants $K_1$ and $K_2$ independent of $N$, which ends the proof.
}
\end{proof}

\bibliographystyle{plain}
\bibliography{bayesian}

\begin{thebibliography}{10}

\bibitem{Amsallem2012nonlinear}
D.~Amsallem, M.~J Zahr, and C.~Farhat.
\newblock Nonlinear model order reduction based on local reduced-order bases.
\newblock {\em International Journal for Numerical Methods in Engineering},
  92(10):891--916, 2012.

\bibitem{Arridge2006}
S.~R. Arridge, J.~P. Kaipio, V.~Kolehmainen, M.~Schweiger, E.~Somersalo,
  T.~Tarvainen, and M.~Vauhkonen.
\newblock Approximation errors and model reduction with an application in
  optical diffusion tomography.
\newblock {\em Inverse Problems}, 22(1):175--195, 2006.

\bibitem{Brooks2011}
S.~Brooks, A.~Gelman, G.~L. Jones, and X.~L. Meng, editors.
\newblock {\em Handbook of {M}arkov chain {M}onte {C}arlo}.
\newblock Chapman \& Hall/CRC Handbooks of Modern Statistical Methods. CRC
  Press, Boca Raton, FL, 2011.

\bibitem{Christen2005markov}
J.~A. Christen and C.~Fox.
\newblock Markov chain monte carlo using an approximation.
\newblock {\em Journal of Computational and Graphical statistics},
  14(4):795--810, 2005.

\bibitem{Cohen2013stability}
A.~Cohen, M.~A Davenport, and D.~Leviatan.
\newblock On the stability and accuracy of least squares approximations.
\newblock {\em Foundations of computational mathematics}, 13(5):819--834, 2013.

\bibitem{Conrad2018parallel}
P.~R Conrad, A.~D Davis, Y.~M Marzouk, N.~S Pillai, and A.~Smith.
\newblock Parallel local approximation mcmc for expensive models.
\newblock {\em SIAM/ASA Journal on Uncertainty Quantification}, 6(1):339--373,
  2018.

\bibitem{Conrad2016JASA}
P.~R Conrad, Y.~M Marzouk, N.~S Pillai, and A.~Smith.
\newblock Accelerating asymptotically exact mcmc for computationally intensive
  models via local approximations.
\newblock {\em Journal of the American Statistical Association},
  111(516):1591--1607, 2016.

\bibitem{Cui2014data}
T.~Cui, Y.~M. Marzouk, and K.~Willcox.
\newblock Data-driven model reduction for the {B}ayesian solution of inverse
  problems.
\newblock {\em International Journal for Numerical Methods in Engineering},
  102(5):966--990, 2015.

\bibitem{Efendiev2006preconditioning}
Y.~Efendiev, T.~Hou, and W.~Luo.
\newblock Preconditioning markov chain monte carlo simulations using
  coarse-scale models.
\newblock {\em SIAM Journal on Scientific Computing}, 28(2):776--803, 2006.

\bibitem{Evans+Stark2002}
S.~N. Evans and P.~B. Stark.
\newblock Inverse problems as statistics.
\newblock {\em Inverse Problems}, 18(4):R55--R97, 2002.

\bibitem{Frangos+Marzouk+Willcox2010}
M.~Frangos, Y.~Marzouk, K.~Willcox, and B.~van Bloemen~Waanders.
\newblock Surrogate and reduced-order modeling: a comparison of approaches for
  large-scale statistical inverse problems.
\newblock {\em Biegler, L. and Biros, G. and Ghattas, O. and Heinkenschloss, M.
  and Keyes, D. and Mallick, B. and Marzouk, Y. and Tenorio, L. and van Bloemen
  Waanders, B. and Willcox, K. editors, Computational Methods for Large Scale
  Inverse Problems and Uncertainty Quantification, John Wiley \& Sons, UK},
  pages 123--149, 2010.

\bibitem{Galbally+Fidkowski+Willcox+Ghattas2010}
D.~Galbally, K.~Fidkowski, K.~Willcox, and O.~Ghattas.
\newblock Non-linear model reduction for uncertainty quantification in
  large-scale inverse problems.
\newblock {\em International Journal for Numerical Methods in Engineering},
  81(12):1581--1608, 2010.

\bibitem{ghanem_and_spanos_book}
R.~G. Ghanem and P.~D. Spanos.
\newblock {\em Stochastic finite elements: a spectral approach}.
\newblock Springer-Verlag, New York, 1991.

\bibitem{Hadigol2018least}
M.~Hadigol and A.~Doostan.
\newblock Least squares polynomial chaos expansion: A review of sampling
  strategies.
\newblock {\em Computer Methods in Applied Mechanics and Engineering},
  332:382--407, 2018.

\bibitem{Hampton2018practical}
J.~Hampton, H.~R Fairbanks, A.~Narayan, and A.~Doostan.
\newblock Practical error bounds for a non-intrusive bi-fidelity approach to
  parametric/stochastic model reduction.
\newblock {\em Journal of Computational Physics}, 368:315--332, 2018.

\bibitem{Jin2008fast}
B.~Jin.
\newblock Fast {B}ayesian approach for parameter estimation.
\newblock {\em International Journal for Numerical Methods in Engineering},
  76(2):230--252, 2008.

\bibitem{Kaipio+Somersalo2005}
J.~P. Kaipio and E.~Somersalo.
\newblock {\em Statistical and {C}omputational {I}nverse {P}roblems}, volume
  160.
\newblock Springer, 2005.

\bibitem{kennedy2001}
Marc~C Kennedy and Anthony O'Hagan.
\newblock Bayesian calibration of computer models.
\newblock {\em Journal of the Royal Statistical Society: Series B (Statistical
  Methodology)}, 63(3):425--464, 2001.

\bibitem{le2010spectral}
Olivier Le~Ma{\^\i}tre and Omar~M Knio.
\newblock {\em Spectral methods for uncertainty quantification: with
  applications to computational fluid dynamics}.
\newblock Springer Science \& Business Media, 2010.

\bibitem{Li+Marzouk2014SISC}
J.~Li and Y.~M Marzouk.
\newblock Adaptive construction of surrogates for the bayesian solution of
  inverse problems.
\newblock {\em SIAM Journal on Scientific Computing}, 36(3):A1163--A1186, 2014.

\bibitem{Lieberman+Willcox+Ghattas2010}
C.~Lieberman, K.~Willcox, and O.~Ghattas.
\newblock Parameter and state model reduction for large-scale statistical
  inverse problems.
\newblock {\em SIAM Journal on Scientific Computing}, 32(5):2523--2542, 2010.

\bibitem{Lin+Xu2007}
Y.~Lin and C.~Xu.
\newblock Finite difference/spectral approximations for the time-fractional
  diffusion equation.
\newblock {\em Journal of Computational Physics}, 225(2):1533--1552, 2007.

\bibitem{lu2015JCP}
F.~Lu, M.~Morzfeld, X.~Tu, and A.~J Chorin.
\newblock Limitations of polynomial chaos expansions in the bayesian solution
  of inverse problems.
\newblock {\em Journal of Computational Physics}, 282:138--147, 2015.

\bibitem{Ma+Zabaras2009}
X.~Ma and N.~Zabaras.
\newblock An efficient {B}ayesian inference approach to inverse problems based
  on an adaptive sparse grid collocation method.
\newblock {\em Inverse Problems}, 25:035013, 2009.

\bibitem{manzoni2016accurate}
A.~Manzoni, S.~Pagani, and T.~Lassila.
\newblock Accurate solution of {B}ayesian inverse uncertainty quantification
  problems combining reduced basis methods and reduction error models.
\newblock {\em SIAM/ASA Journal on Uncertainty Quantification}, 4(1):380--412,
  2016.

\bibitem{Marzouk+Najm2009}
Y.~M. Marzouk and H.~N. Najm.
\newblock Dimensionality reduction and polynomial chaos acceleration of
  {B}ayesian inference in inverse problems.
\newblock {\em Journal of Computational Physics}, 228(6):1862--1902, 2009.

\bibitem{Marzouk+Najm+Rahn2007}
Y.~M. Marzouk, H.~N. Najm, and L.~A. Rahn.
\newblock Stochastic spectral methods for efficient {B}ayesian solution of
  inverse problems.
\newblock {\em Journal of Computational Physics}, 224(2):560--586, 2007.

\bibitem{Marzouk+Xiu2009}
Y.~M. Marzouk and D.~Xiu.
\newblock A stochastic collocation approach to {B}ayesian inference in inverse
  problems.
\newblock {\em Communications in Computational Physics}, 6:826--847, 2009.

\bibitem{Migliorati2013approximation}
G.~Migliorati, F.~Nobile, E.~von Schwerin, and R.~Tempone.
\newblock Approximation of quantities of interest in stochastic pdes by the
  random discrete {L}\^{}2 projection on polynomial spaces.
\newblock {\em SIAM Journal on Scientific Computing}, 35(3):A1440--A1460, 2013.

\bibitem{Narayan2014}
A.~Narayan, J.~D. Jakeman, and T.~Zhou.
\newblock A {C}hristoffel function weighted least squares algorithm for
  collocation approximations.
\newblock {\em Mathematics of Computation}, 86(306):1913--1947, 2017.

\bibitem{Ng2012multi}
Leo Wai-Tsun Ng and M.~Eldred.
\newblock Multifidelity uncertainty quantification using non-intrusive
  polynomial chaos and stochastic collocation.
\newblock In {\em 53rd AIAA/ASME/ASCE/AHS/ASC Structures, Structural Dynamics
  and Materials Conference 20th AIAA/ASME/AHS Adaptive Structures Conference
  14th AIAA}, page 1852, 2012.

\bibitem{Palar2016multi}
P.~S. Palar, T.~Tsuchiya, and Geoffrey~T. Parks.
\newblock Multi-fidelity non-intrusive polynomial chaos based on regression.
\newblock {\em Computer Methods in Applied Mechanics and Engineering},
  305:579--606, 2016.

\bibitem{Peherstorfer2016multifidelity}
B.~Peherstorfer, T.~Cui, Y.~Marzouk, and K.~Willcox.
\newblock Multifidelity importance sampling.
\newblock {\em Computer Methods in Applied Mechanics and Engineering},
  300:490--509, 2016.

\bibitem{Peherstorfer2016survey}
B.~Peherstorfer, K.~Willcox, and M.~Gunzburger.
\newblock Survey of multifidelity methods in uncertainty propagation,
  inference, and optimization.
\newblock {\em Preprint}, pages 1--57, 2016.

\bibitem{Podlubny1999}
I.~Podlubny.
\newblock {\em Fractional differential equations}, volume 198 of {\em
  Mathematics in Science and Engineering}.
\newblock Academic Press Inc., San Diego, CA, 1999.

\bibitem{Rasmussen2003}
Carl~Edward Rasmussen, JM~Bernardo, MJ~Bayarri, JO~Berger, AP~Dawid,
  D~Heckerman, AFM Smith, and M~West.
\newblock Gaussian processes to speed up hybrid {M}onte {C}arlo for expensive
  {B}ayesian integrals.
\newblock In {\em Bayesian Statistics 7}, pages 651--659, 2003.

\bibitem{Stuart2010}
A.~M. Stuart.
\newblock Inverse problems: a {B}ayesian perspective.
\newblock {\em Acta Numerica}, 19(1):451--559, 2010.

\bibitem{stuart+teckentrup2016}
A.~M Stuart and A.~Teckentrup.
\newblock Posterior consistency for {G}aussian process approximations of
  {B}ayesian posterior distributions.
\newblock {\em Mathematics of Computation}, 87(310):721--753, 2018.

\bibitem{Tang2014discrete}
T.~Tang and T.~Zhou.
\newblock On discrete least-squares projection in unbounded domain with random
  evaluations and its application to parametric uncertainty quantification.
\newblock {\em SIAM Journal on Scientific Computing}, 36(5):A2272--A2295, 2014.

\bibitem{xiu2010book}
D.~Xiu.
\newblock {\em Numerical methods for stochastic computations: a spectral method
  approach}.
\newblock Princeton University Press, 2010.

\bibitem{yan+guo2015}
L.~Yan and L.~Guo.
\newblock Stochastic collocation algorithms using $l_1$-minimization for
  {B}ayesian solution of inverse problems.
\newblock {\em SIAM Journal on Scientific Computing}, 37(3):A1410--A1435, 2015.

\bibitem{Yan+Zhang2017IP}
L.~Yan and Y.X. Zhang.
\newblock Convergence analysis of surrogate-based methods for bayesian inverse
  problems.
\newblock {\em Inverse Problems}, 33(12):125001, 2017.

\bibitem{Zahr2015progressive}
M.~J Zahr and C.~Farhat.
\newblock Progressive construction of a parametric reduced-order model for
  pde-constrained optimization.
\newblock {\em International Journal for Numerical Methods in Engineering},
  102(5):1111--1135, 2015.

\bibitem{Zhou2015weighted}
T.~Zhou, A.~Narayan, and D.~Xiu.
\newblock Weighted discrete least-squares polynomial approximation using
  randomized quadratures.
\newblock {\em Journal of Computational Physics}, 298:787--800, 2015.

\end{thebibliography}

\end{document}